\documentclass[11pt]{article}
\usepackage{fullpage}
\usepackage{epsfig}
\usepackage{epstopdf}
\usepackage{stmaryrd}
\usepackage{color}
\usepackage{amsmath, amsthm}
\usepackage{amsfonts,amssymb}
\usepackage{dsfont}
\usepackage[utf8]{inputenc}
\usepackage{hyperref}
\usepackage{graphicx, subfigure, psfrag}
\usepackage{enumerate}
\usepackage{cite}

\newtheorem{thm}{Theorem}[section]
\newtheorem{pro}[thm]{Proposition}

\newtheorem{lem}[thm]{Lemma}
\newtheorem{fact}[thm]{Fact}
\newtheorem{rem}[thm]{Remark}
\newtheorem{de}[thm]{Definition}

\def \beq{\begin{equation}}
\def \eq{\end{equation}}
\def\eqd{\,{\buildrel (d) \over =}\,} 
\def\cvd{\,{\buildrel (d) \over \longrightarrow}\,}
\def\cvp{\,{\buildrel \p \over \longrightarrow}\,}
\def\Var{{\rm Var}\,}

\def \eref#1{(\ref{#1})}

\newcommand{\T}{\mathcal{T}}
\newcommand{\V}{\mathcal{V}}
\newcommand{\I}{\mathcal{I}}

\newcommand{\f}{\mathfrak{f}}
\newcommand{\F}{\mathcal{F}^0}
\newcommand{\D}{\mathcal{D}}
\newcommand{\U}{\mathcal{U}}
\newcommand{\W}{\mathcal{W}} 
\newcommand{\p}{\mathbb{P}}
\newcommand{\N}{\mathbb{N}}
\newcommand{\C}{\mathcal{C}}
\newcommand{\E}{\mathbb{E}}
\newcommand{\R}{\mathbb{R}}
\newcommand{\pt}{\mathcal{CT}}
\newcommand{\ft}{\mathcal{FT}}
\newcommand{\up}{\mathbb{U}^{\nu}}
\newcommand{\qp}{\mathbb{H}^{\nu}}
\newcommand{\Dir}{\mathrm{Dir}}
\newcommand{\dm}{\mathrm{dim}_H}
\newcommand{\ms}{\mathcal{M}_S(V_2^*)}

\newcommand{\overbar}[1]{\mkern 1.5mu\overline{\mkern-1.5mu#1\mkern-1.5mu}\mkern 1.5mu}

\begin{document}

\title{Representations of stack triangulations in the plane}
\author{Thomas Selig \\ LaBRI, CNRS \\ Université Bordeaux 1 }
\maketitle

\begin{abstract}
Stack triangulations appear as natural objects when defining an increasing family of triangulations by successive additions of vertices. We consider two different probability distributions for such objects. We represent, or ``draw" these random stack triangulations in the plane $\R^2$ and study the asymptotic properties of these drawings, viewed as random compact metric spaces. We also look at the occupation measure of the vertices, and show that for these two distributions it converges to some random limit measure.
\end{abstract}

{\bf Keywords:} Stack triangulations; Occupation measure; Limit Theorem.

\section{Introduction}\label{Intro}
Consider a rooted triangulation of the plane, and some finite face $\f$, say $ABC$, of this triangulation. We insert a vertex $M$ in $\f$, and add the three edges $AM$, $BM$, $CM$ to the original triangulation. We obtain a triangulation with two faces more than the original triangulation (the face $\f$ has been replaced by three new faces). Thus, starting from a single rooted triangle, after $k$ such insertions, we get a triangulation with $k$ internal vertices, that is which aren't vertices of the original rooted triangle, and $2k+1$ finite faces. The set of triangulations with $k$ internal vertices which can be reached through this growing procedure is denoted $\Delta_k$. We call such triangulations \emph{stack triangulations}. Note that through this construction we do not obtain the set of all rooted triangulations. This iterative process is demonstrated in Figure \ref{ST}.

\begin{figure}[h]
\centerline{\includegraphics{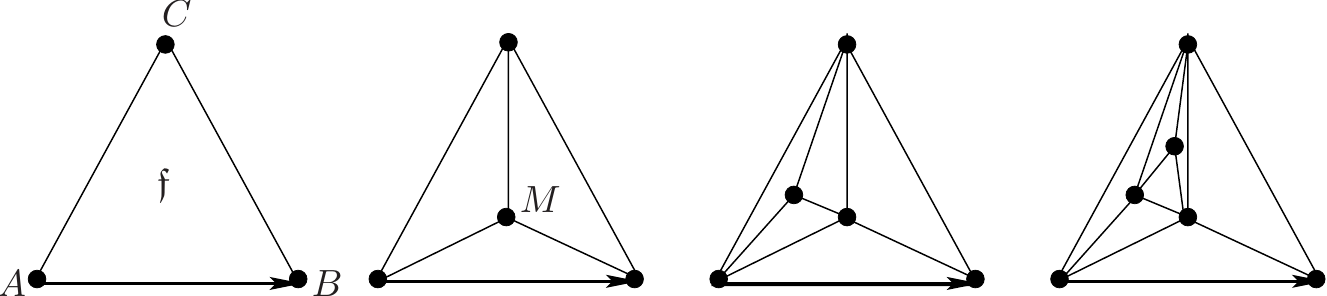}}
\caption{Iterative construction of a stack triangulation}
\label{ST}
\end{figure}

We endow the set $\Delta_k$ with two natural probability distributions:
\begin{itemize}
\item The first is the uniform distribution $\mathbb{U}_k^{\Delta}$.
\item The second distribution $\mathbb{H}_k^{\Delta}$ is the distribution induced by the above construction where at each step the face in which we insert the vertex is chosen uniformly at random among all finite faces, independently from the past.
\end{itemize}

\subsection{The object of our study}
In this paper, rather than look at stack triangulations as maps, that is up to homeomorphism, we look at particular representations, or drawings, of such objects in the plane, and the geometrical properties of such representations. That is, at each insertion of a new vertex, we draw the line segments corresponding to the edges added. We call such representations \textit{compact triangulations}, and view them as compact subspaces of $\R^2$. The main difference is that while maps are graphs drawn in the plane, they are considered only up to homeomorphism, whereas we are interested in the actual representation. We are rather informal here, but will give formal definitions later in the paper, in Section \ref{formal_def}. 

We take $A=(0,0)$, $B=(1,0)$, $C = e^{i\frac{\pi}{3}}$ (identifying $\mathbb{C}$ and $\R^2$) to be the three points representing the initial rooted triangle, with $(A,B)$ its root. We start with $T_0 = T = [AB] \cup [BC] \cup [CA]$, and set $\D_0 = \{ T_0 \}$. We denote by $\tilde{T_0}$ the \textit{filled} triangle $T_0$, that is the union of $T_0$ and of the finite connected component of $\R^2 \setminus T_0$. At time $1$, we insert a point $M$ somewhere in $\tilde{T_0} \setminus T_0$. We then define $T_1(M) = T_0 \cup [AM] \cup [BM] \cup [CM] $, and set
\[ \D_1 =  \{ T_1(M); \ M \in \tilde{T_0} \setminus T_0\}.\]

Now let $T_1(M) \in \D_1$ for some $M \in \tilde{T_0} \setminus T_0$. Write $T_1^{(1)}(M) = [AB] \cup [BM] \cup [MA]$, $T_1^{(2)}(M) = [BC] \cup [CM] \cup [MB]$, $T_1^{(3)}(M) = [CA] \cup [AM] \cup [MC]$, and similarly, $\tilde{T_1}^{(i)}(M)$ for the corresponding filled triangles. At time 2, we insert a point $N$ somewhere in one of the $\tilde{T_1}^{(i)}(M) \setminus T_1^{(i)}(M)$'s. We then define
\[ T_2(M,N) := T_1(M) \cup [XN] \cup [YN] \cup [MN],\]
where $(X,Y) = \left\lbrace
\begin{array}{cc}
(A,B) & \mbox{if } N \in \tilde{T_1}^{(1)}(M) \setminus T_1^{(1)}(M)\\
(B,C) & \mbox{if } N \in \tilde{T_1}^{(2)}(M) \setminus T_1^{(2)}(M) \\
(C,A) & \mbox{if } N \in \tilde{T_1}^{(3)}(M) \setminus T_1^{(3)}(M)
\end{array}
\right. $, and finally set
\[ \D_2 = \left\lbrace T_2(M,N); \ M \in \tilde{T_0} \setminus T_0 \mbox{ and } N \in \bigcup_{i=1}^3 \left( \tilde{T_1}^{(i)}(M) \setminus T_1^{(i)}(M) \right) \right\rbrace .\]
Figure \ref{iter} illustrates this initial construction.

\begin{figure}[!h]
\centerline{\includegraphics{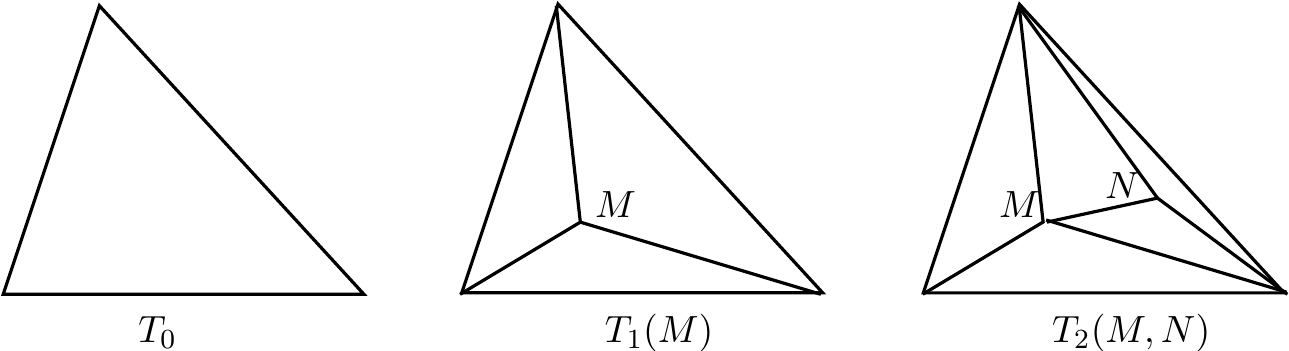}}
\caption{\label{iter} Construction of $\D_0,\D_1$ and $\D_2$}
\end{figure}

Iterating this construction by choosing at each step some triangular face of our drawing and splitting it, we obtain representations of stack triangulations in the plane, and call these compact triangulations. Denote $\D_k$ the set of such objects with $k$ internal vertices (that is, after $k$ successive insertions of vertices), and for $m \in \D_k$ write $\V(m)$ for its set of internal vertices (viewed as a set of points in the plane). Finally, for $m \in \D_k$ we define the occupation measure of $m$ by
\beq\label{def_OM1}
\mu(m) := \frac1k \sum_{x \in \V(m)} \delta_x,
\eq
where $\delta_x$ stands for the Dirac mass at $x$.

We are interested here in the case where the successive insertions of vertices are done at random. We suppose that all random variables in this paper are defined on some probability space $( \Omega, \mathcal{F}, \p)$. We denote $\E$ the expected value, and $\Var$ the variance. We consider a probability distribution $\nu$ on $\R_+^3$ such that if $P=(P_1,P_2,P_3)$ has law $\nu$ then a.s. $P_i > 0$ for all $i \in \{1,2,3 \}$  and $P_1+P_2+P_3=1$ \footnote{This is the notion of splitting law, defined formally in Section \ref{formal_def}}. We suppose that each insertion of a vertex $M$ in a face $QRS$ is done according to $\nu$, that is we take $M$ to have barycentric coordinates $(Q,P_1),(R,P_2),(S,P_3)$ where $P=(P_1,P_2,P_3)$ has law $\nu$, independently from all previous insertions. Now the two distributions $\mathbb{U}_k^{\Delta}$ and $\mathbb{H}_k^{\Delta}$ on $\Delta_k$ introduced at the start of the section induce probability distributions $\up_k$ and $\qp_k$ on $\D_k$. In words, they are the distributions of the drawings of stack triangulations with distribution $\mathbb{U}_k^{\Delta}$ and $\mathbb{H}_k^{\Delta}$, where the insertions of vertices are made according to $\nu$, independently from each other, and independently from the choice of the underlying stack triangulation. The object of the paper is to study the asymptotic behaviour of these two distributions.

\subsection{Outline of the paper}
In Section \ref{sec_def}, we formally define compact triangulations. We then enrich the classical bijection between stack triangulations and ternary trees (see for instance \cite{AM}, Proposition 1) to encode compact triangulations. For this, in Section \ref{ad_cod}, we introduce the notion of coordinate-labelled ternary trees. These are ternary trees with labels at each vertex, which code compact triangulations via a bijection we establish in Theorem \ref{bij_drawing_markedtree}. We end the section by formally defining the distributions $\up_k$ and $\qp_k$ on $\D_k$.

In Section \ref{unif}, we study the asymptotic behaviour of the uniform distribution $\up_n$ as $n \rightarrow \infty$. The main results are:
\begin{itemize}
\item The weak convergence of the occupation measure as defined in \eref{def_OM1} towards a Dirac mass at a random position (Theorem \ref{UOM}).
\item The weak convergence of the distribution $\up_n$ towards a distribution on compact subspaces of $\R^2$ (Theorem \ref{conv_draw}).
\end{itemize}
Section \ref{sectionUOM} is thus dedicated to the statement and proof of Theorem \ref{UOM}. The statement is split into two parts. The first part states the convergence in distribution of an internal vertex of $m_n$ chosen uniformly at random, where $m_n$ has distribution $\up_n$, to some limit vertex. Though this is weaker than the second part, it is a key ingredient in its proof, and thus we choose to state it separately. In Section \ref{UD} we introduce the notion of local convergence for trees (Definition \ref{loc_conv_def}). In Theorem \ref{cont_thm}, we show that the bijection established in Theorem \ref{bij_drawing_markedtree} which maps a coordinate-labelled tree to a compact triangulation has a property of continuity with respect to this topology of local convergence, from which we infer Theorem \ref{conv_draw}.

Finally, in Section \ref{incr}, we study the asymptotic behaviour of the occupation measure under $\qp_n$. The key ingredient is Poisson-Dirichlet fragmentation, which allows us to view the trees corresponding to the compact triangulations via Theorem \ref{bij_drawing_markedtree} as the underlying tree of a certain fragmentation tree (Theorem \ref{frag_tree}). We then show the convergence of the occupation measure as defined in \eref{def_OM1} to some (random) limit measure $\mu$ (Theorem \ref{IOM}). In Section \ref{IMP} we study the properties of $\mu$. We show (Proposition \ref{at_part}) that a.s. $\mu$ has no atom, and that it is supported on a set whose Hausdorff dimension is at most $\dfrac{2}{3 \log (3)}$ (Theorem \ref{HDim}). 

\subsection{Literature and motivation}

Motivation for this work stems from the paper by Bonichon et al. \cite{Bon}, in which the authors look at convex straight line drawings of triangulations, and establish bounds for the minimal grid size necessary for these drawings, with the constraint that all vertices are located at integer grid points. More precisely, they show that to draw any triangulation with $n$ faces, a grid of size $(n-C) \times (n-C)$ (for some constant $C$) is sufficient, giving a constructive proof of this result by establishing an algorithm for drawing any triangulation. The aim of this paper if to provide an answer to the question: what do these drawings look like? 

More specifically, we aim to explore an approach for the convergence of maps which differs from the traditional combinatorial one. Indeed, maps are embeddings of graphs, but in the combinatorial approach these are viewed up to homeomorphism, and equipped with the graph distance, that is every edge is given the same length. Concerning this approach, we cite the groundbreaking work by Schaeffer in his thesis \cite{Sch}, where he establishes a crucial bijection between maps and a class of labelled trees, as well as the more recent work by LeGall \cite{LeGall} and Miermont \cite{Mier}, who showed (separately, using different techniques) that uniform quandrangulations with $n$ faces, renormalised so that every edge has length $C.n^{\frac14}$ (for some constant $C$), converge in distribution to a continuous limit object called the Brownian map.
In this paper however, we look at the convergence of the embeddings themselves, viewed as (random) compact spaces. This approach is analogous to the work of Curien and Kortchemski \cite{CK}. In this paper, the authors showed the universality of the Brownian triangulation introduced by Aldous \cite{Ald2}, in that is the limit of a number of discrete families called non-crossing plane configurations, such as dissections, triangulations, and non-crossing trees of the regular $n$-gon. As mentioned, Curien and Kortchemski view non-crossing plane configurations as random compact subspaces of the unit disk, and it is these compact spaces which converge to the limit object.

In this paper, we also study the asymptotic behaviour of the occupation measure, as defined in \eref{def_OM1}. Similar work includes the paper by Fekete \cite{Fek} on branching random walks. In this paper, he considers branching random walks where the underlying tree is a binary search tree (this is related to our distribution $\qp_n$ in this paper). He shows that the occupation  measure converges weakly to a limit measure which is deterministic. More work concerning the study of random measures similar to ours can be found in \cite{Ald1}. In this paper, Aldous proposes a natural model for random continuous ``distributions of mass", called the \textit{Integrated super-Brownian Excursion} (ISE), which is the (random) occupation measure of the Brownian snake with lifetime process the normalised Brownian excursion. ISE is defined using random branching structures, and appears to be the continuous limit of occupation measures of several discrete structures.

Finally, let us mention that the combinatorial aspect of stack triangulations has been extensively studied, notably by Albenque and Marckert \cite{AM}, and their paper will therefore be of great use to us. The authors studied both the uniform distribution $\mathbb{U}^{\Delta}_k$ and the other distribution$\mathbb{H}^{\Delta}_k$. More precisely, they showed that:
\begin{itemize}
\item for the topology of local convergence, $\mathbb{U}^{\Delta}_n$ converges weakly to a distribution on the set of infinite maps. 
\item For the Gromov-Hausdorff distance, with the normalising factor $n^{\frac12}$, a map with the uniform distribution $\mathbb{U}^{\Delta}_n$ converges weakly to the continuum random tree introduced by Aldous \cite{CRT}
\item Under the distribution $\mathbb{H}^{\Delta}_n$, the distance between random points rescaled by $\frac{6}{11} \log n$ converges to $1$ in probability.
\end{itemize}

\section{Compact triangulations and encoding with trees}\label{sec_def}

In this section we code compact triangulations, that is the representations of triangulations in the plane, by some labelled trees. There are two main ideas in this coding. First there is the combinatorial bijection between the discrete objects: stack triangulations (viewed up to homeomorphism) and ternary trees. There is a well known bijection which maps internal vertices of the triangulation to internal nodes of the tree and faces of the triangulation to leaves of the tree (see for instance \cite{AM} Proposition 1 and references therein). We then enrich this bijection to include the drawing of the triangulation by adding labels to the tree: these labels correspond to the barycentric coordinates of the vertices of the triangulation.

\subsection{Compact triangulations}\label{formal_def}

Here we build formally the set $\D_k$ of compact triangulations with $k$ internal vertices. The construction is done by induction, and is similar to the construction of stack triangulations. This allows us to observe the tree-like structure of these objects. During the construction, we will define the various notions necessary for the encoding discussed above.
Set as in the introduction $A=(0,0)$, $B=(1,0)$, $C = e^{i\frac{\pi}{3}}$ to be the three points of the original triangle, and define $T = [AB] \cup [BC] \cup [CA]$. Now define $\D_0 = \{ T \}$, and set $\V(T) = \emptyset$. The set $\V(T)$ will be the set of internal nodes of $T$.
Now assume we have constructed $\D_k$ for some $k \geq 0$, such that $\D_k$ is a set of compact subspaces of $\R^2$ and any $m \in \D_k$ satisfies the following properties:
\begin{enumerate}
\item The compact space $m$ is the union of line segments in the plane.
\item There are exactly $2k+1$ finite connected components of $\R^2 \setminus m$, and these are all interiors of triangles. Let $\F(m)$ be the set of these connected components, and call the elements of $\F(m)$ \textit{faces} of $m$. For $f \in \F(m)$ we define $(A_{f},B_{f},C_{f})$ as the three points of the triangle $f$. We can in fact define these points non ambiguously as follows.
\begin{itemize}
\item $X_f = X$ for $X \in \{A,B,C \}$, if $f$ is the interior of the original triangle $T$.
\item If a triangle $f$ is split into three triangles $f_1,f_2,f_3$ by adding a point $M$ in its interior, and $f$ is defined by the three points $A_{f},B_{f},C_{f}$, then $M=A_{f_1} = B_{f_2} = C_{f_3}$ with the other two vertices of each triangle $f_i$ unchanged (that is, $B_{f_1} = B_{f}$, $C_{f_1}=C_{f}$ and so forth). This is illustrated in Figure \ref{trg_order} below.
\end{itemize}
\item Finally assume that for any $m \in \D_k$ we have defined a set $\V(m)$ of $k$ points of $\tilde{T}$, which are the $k$ points inserted at each step of the construction of $m$.
\end{enumerate}

Note that these properties are all satisfied for $k=0$. 
\begin{figure}[!h]
\centerline{\includegraphics{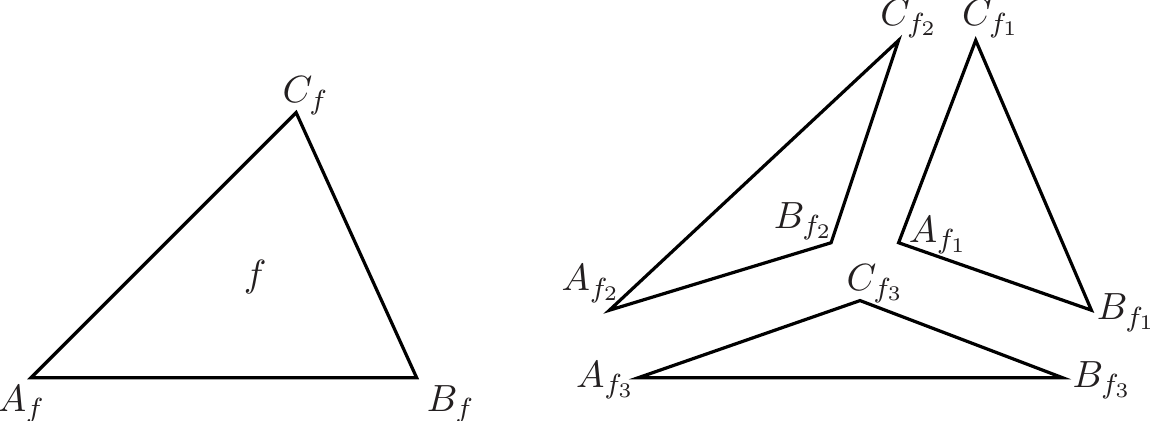}}
\caption{Ordering the vertices of a triangle}
\label{trg_order}
\end{figure}

We now construct the set $\D_{k+1}$. First, let
\[ \dot{\D_k} = \lbrace (m,f) ; \ m \in \D_k, f \in \F(m) \rbrace \]
be the set of compact triangulations with a marked face. Define a map $\I$ from $\dot{\D_k}$ onto the set of compact subspaces of $\R^2$ as follows. Let $(m,f) \in \dot{\D_k}$, and let $(A_{f},B_{f},C_{f})$ be the three (ordered) points of $f$. For any point $M$ in the face $f$ we define
\[m' = \I_M(m,f) := m \cup \{ [A_{f}M], [B_{f}M], [C_{f}M] \}, \]
that is the space $m$ with those three new lines added, connecting the points of the face $f$ with the inserted vertex $M$. The map $\I_M$ is illustrated in Figure \ref{insertion}. We see that there are exactly $2k+3$ finite connected components of $\R^2 \setminus m'$, and these are all interiors of triangles (we have replaced one of them, $f$, by three new ones). We also set 
\beq\label{def_V}
\V(m') = \V(m) \cup \{M \},
\eq
and thus the set $\V(m')$ is a set of $k+1$ points of $\tilde{T}$: it is the set of the internal vertices which define the faces of $m'$. Finally, we can define 
\[ \D_{k+1} :=
 \left\lbrace \I_M(m,f); \, (m,f) \in \dot{\D_k}, M \in f\right\rbrace \]
to be the image of this map. In words, it is the set of `` drawings " of stack triangulations with $k$ internal vertices, with edges included.
\begin{figure}[!h]
\centerline{\includegraphics{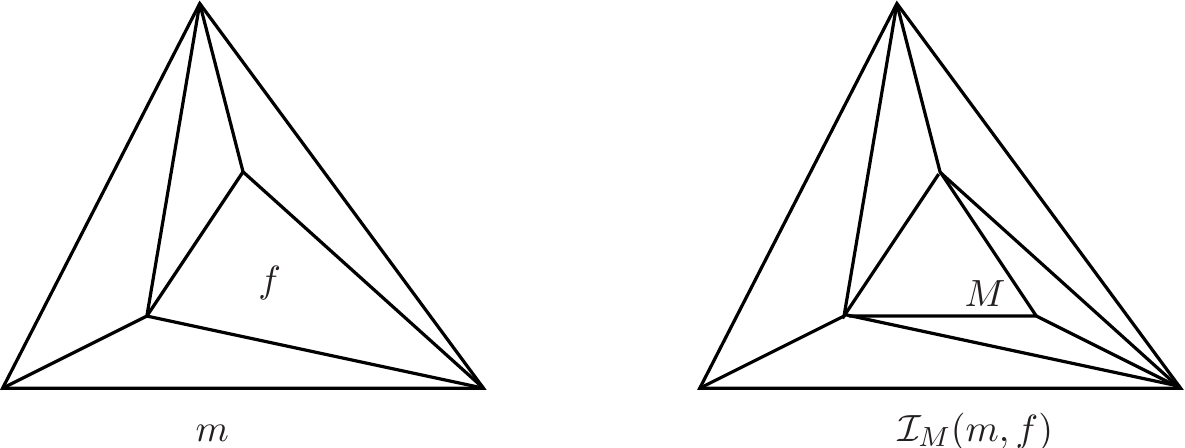}}
\caption{The insertion map $\I$}
\label{insertion}
\end{figure}

\begin{de}
Let $k\geq 0$. For $m \in \D_k$, we call the elements of $\V(m)$ (where $\V(m)$ is defined step by step by \eref{def_V}) the \emph{internal vertices} of $m$. The set $\D_k$ is called the set of compact triangulations with $k$ internal vertices. Finally, we denote
\[ \D = \bigcup_{k \geq 0} \D_k \]
the set of compact triangulations.
\end{de}

\begin{de}\label{defOM}
Let $m \in \D$. We define the \emph{occupation measure} of $m$ by
\beq\label{def_OM}
\mu(m) = \frac{1}{\vert \V(m) \vert} \sum_{x \in \V(m)} \delta_x.
\eq
This is a probability measure in $\R^2$.
\end{de}

Note that $\D_k$ is a set of compact subspaces of $\R^2$. We aim to introduce some probability laws on these sets (as explained in the introduction), and for this we need to equip them with a $\sigma$-field. We first recall the definition of the Hausdorff distance for compact spaces.

\begin{de}
Let $(E,d)$ be a compact metric space. For $A \subseteq E$ and $\varepsilon>0$, define the \emph{$\varepsilon$-neighbourhood of $A$} as the set of points of $E$ whose distance to $A$ is less than $\varepsilon$, that is
\[ V^{ \varepsilon }(A) = \lbrace x \in E, \, d(x,A) < \varepsilon \rbrace. \]
Then for two compact sets $A,B \subseteq E$, the Hausdorff distance between $A$ and $B$ is defined by
\[ d_H(A,B) = \inf \lbrace \varepsilon > 0, \, A \subseteq V^{ \varepsilon }(B) \mbox{ and } B \subseteq V^{ \varepsilon }(A) \rbrace.\]
This defines a distance on the set of compact subspaces of $E$.
\end{de}

We equip the space of compact subspaces of $\tilde{T}$ with the Hausdorff distance. It is a well-known topological fact that this makes it a complete metric space (see for instance \cite{BBI} Section 7.3.1 p. 252). In fact, $(\D_k,d_H)$ is a compact metric space. We equip the sets $\D_k$ with the corresponding Borel $\sigma$-algebra.

\subsection{Encoding with properly marked trees}\label{ad_cod}

We now encode compact triangulations by certain labelled trees. We begin with the purely combinatorial aspect. Let
\[ \mathcal{W} := \bigcup_{n \geq 0} \N^n \]
be the set of all words on $\N = \lbrace 1,2,... \rbrace $, and by convention set $ \N^0 = \lbrace \emptyset \rbrace $.\\
If $u = (u_1,...,u_n) \in \mathcal{W}$ we write $ \vert u \vert = n $ and call this the \textit{height} of $u$. Also, if we take two words $u = (u_1,...,u_n) \, ,v = (v_1,...,v_m) \in \mathcal{W}$, we write $ uv = (u_1,...,u_n,v_1,...,v_m) $ for the concatenation of $u$ and $v$. By convention $u \emptyset = \emptyset u = u$. A \emph{planar tree} is a subset $t \subseteq \W$ such that 
\begin{enumerate}
\item $ \emptyset \in t $.
\item If $u(j) \in t$ for some $u \in \W$ and $j \in \N$, then $u \in t$. \\The notation $u(j)$ is used here to mark the fact that we are concatenating the \emph{words} $u$ and $(j)$, the latter being written so as to differentiate it from the letter $j$.
\item For every $u \in t$ there exists $k_u(t) \in \N$ such that $u(j) \in t$ if and only if $1 \leq j \leq k_u(t)$.
\end{enumerate}

The integer $k_u(t)$ corresponds to the number of children (or descendants) of $u$ in $t$.

We will denote by $\U$ the set of planar trees. If $t \in \U$ is a planar tree, its height $h(t)$ is defined by $h(t) := \sup \lbrace \vert u \vert; \ u \in t \rbrace \in \llbracket 0, \, \infty \rrbracket $. If $u \in t$ has no child (i.e. $k_u(t) = 0$) we say that $u$ is a \emph{leaf} of $t$. Any vertex which isn't a leaf is called an \emph{internal node} of $t$. We denote $t^0$ the set of internal nodes of a tree $t$. If $t$ is a tree, and $u,v$ are in $t$, we write $u \wedge v$ for the \emph{highest common ancestor} of $u$ and $v$, i.e. the element of maximal height of the set $ \lbrace w \in t ; \ \exists (u',v'), \, u=wu' \mbox{ and } v=wv' \rbrace $. If $u \in t$, we let $\theta_u(t) = \lbrace v \in \W ; uv \in t \rbrace $. This is the subtree of $t$ which has $u$ as a root. A ternary tree $t$ is a planar tree such that $ \forall u \in t, \, k_u(t) \in \lbrace 0,3 \rbrace $, i.e. every internal node has exactly three children. 

We denote $\T$ the set of ternary trees, and henceforth we will simply call them trees. We denote $\T^{\infty}$ the infinite complete ternary tree, that is
\[ \T^{\infty} = \bigcup_{n \geq 0} \{1,2,3 \}^n.\]

It is a well known fact that $\T$ is mapped bijectively to the set of stack triangulations. However, compact triangulations contain more information than stack triangulations, since compact triangulations contain the information of where each internal vertex is placed. The additional information will be put at each vertex of the associated ternary tree, and will be the barycentric coordinates of the point associated with $u$, at the time it has been inserted. The idea is thus to associate with a point $M$ its triplet $\C(M)$ of barycentric coordinates with respect to $(A,B,C)$, taken to be with sum equal to $1$. As such $ \C(A)=[1,0,0], \ \C(B)=[0,1,0], \ \C(C)=[0,0,1]$. Equivalently, if the splitting of $T$ is given as in Figure \ref{split}, then $\C(M) = (P_1,P_2,P_3)$, where $(P_1,P_2,P_3)$ are the respective ratios of the areas of the triangles $MBC,AMC,AMB$ over the area of $ABC$.

\begin{figure}[!h]
\centerline{\includegraphics{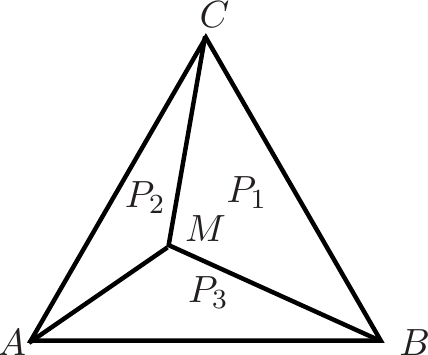}}
\caption{The splitting of a triangle via $P$}
\label{split}
\end{figure}
Write 
\[ V_{2} := \lbrace (x_1, x_2 ,x_3) \in \R^3; \ x_1, x_2 ,x_3 \geq 0 \mbox{ and } x_1+ x_2 +x_3=1 \rbrace\]
for the $3$-dimensional simplex, so that any point $M \in \tilde{T}$ corresponds bijectively to its (normalised) barycentric coordinates in $V_2$. Also, let
 \[ V_{2}^* := \lbrace (x_1, x_2 ,x_3) \in \R^3; \ x_1, x_2 ,x_3 > 0 \mbox{ and } x_1+ x_2 +x_3=1 \rbrace \]
be the $3$-dimensional simplex with its boundary removed.

\begin{de}\label{def_pltrees}
\begin{enumerate}
\item A \emph{fragmentation-labelled} tree is a pair $(t,(P(u))_{u \in t^0})$, $t \in \T$, such that for any $u \in t^0$, $P(u) \in V_2^*$, that is a tree $t$ and a set of triplets $P(u)$ indexed by the internal nodes of $t$. $P(u)$ is called the \emph{splitting triplet at $u$}.
We denote $\ft_n$ the set of fragmentation-labelled trees with $n$ internal vertices, and $\ft = \bigcup \ft_n$.
\item A \emph{coordinate-labelled} tree is a pair $(t,(\lambda(u))_{u \in t})$, $t \in \T$, with labels $\lambda(u)$ at each node $u$ such that: 
\begin{enumerate}[(a)]
\item For each leaf $l$ of the tree $t$, we have $\lambda(l) \in V_2^3$, and write $\C(l) = \lambda(l)$. The elements of $\C(l)$ are called the \emph{coordinates} of $l$.
\item For each internal node $u$ we have $\lambda(u) = (\C(u),P(u))$, with $\C(u) \in V_2^3$, and $P(u) \in V_2^*$. The elements of $\C(u)$ are called the \emph{coordinates} of $u$, and $P(u)$ is called the \emph{splitting triplet} at $u$.
\item The coordinates of the root are $\C(\emptyset)=([1,0,0],[0,1,0],[0,0,1])$.
\item If $\C(u)=(\C_1,\C_2,\C_3)$ and $P(u)=(P_1,P_2,P_3)$ then for $i\in \{1,2,3\}$ we have $\C(u(i))=(\tilde{\C}_j)_{j \in \{1,2,3\} }$ where $\tilde{\C}_j$ is equal to $ C(u).P(u) := P_1\C_1+P_2\C_2+P_3\C_3$ if $j=i$ and $\C_j$ otherwise. This property is illustrated in Figure \ref{loc_label}.
\end{enumerate}
We denote $\pt_n$ the set of coordinate-labelled trees with $n$ internal vertices and likewise $\pt = \bigcup \pt_n$. For $ t \bullet \in \pt$ we will denote $p(t \bullet) \in \T$ the underlying tree.
\end{enumerate}
\end{de}

\begin{figure}[!h]
\centerline{\includegraphics{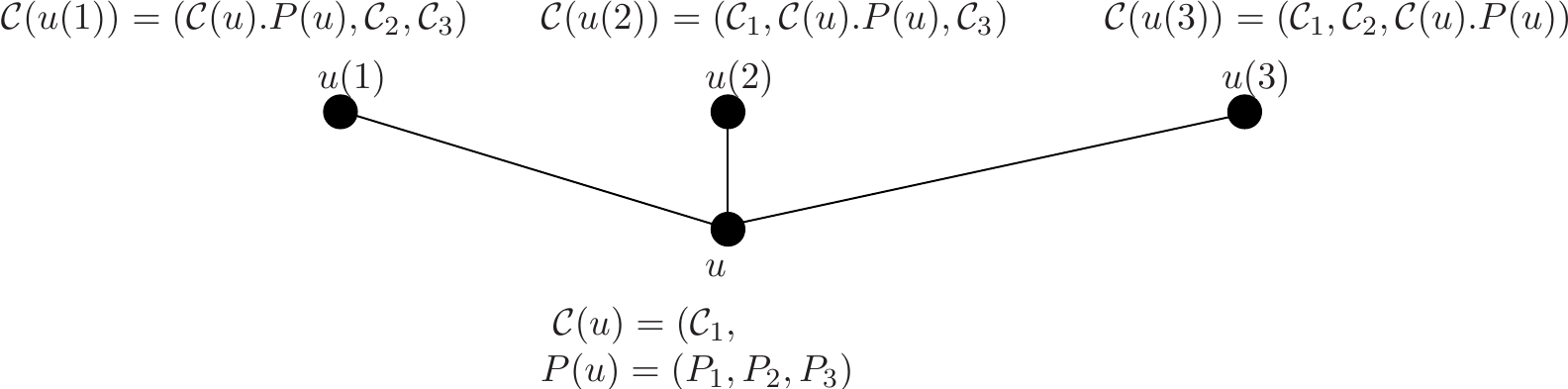}}
\caption{The local labelling rule (d) for coordinate-labelled trees}
\label{loc_label}
\end{figure}

\begin{rem}\label{proper_insertion}
The condition $P(u) \in V_2^*$ (as opposed to $V_2$) means that the insertions of new vertices at each step are \emph{proper} insertions, that is the new point is added in the interior of a face and not on its border. This is crucial for Theorem \ref{bij_drawing_markedtree}.
\end{rem}

\begin{rem}\label{def_Phi}
We can map a fragmentation-labelled tree $(t,(P(u))_{u \in t^0}) \in \ft_k$ to a coordinate-labelled tree $t \bullet \in \pt_k$ by setting $p(t \bullet) = t$, keeping for every internal node $u$ of $t \bullet$ the same splitting triplet $P(u)$ as in $t$, and filling in the remaining triplets of coordinates using rules (c) and (d). This gives us a bijective mapping which we denote $\Phi_k: \ft_k \rightarrow \pt_k$.
\end{rem}

Once more, we aim to define probability distributions on the sets $\ft$ and $\pt$. For this, we introduce a $\sigma$-field on the sets $\ft_k$ and $\pt_k$, with the help of a distance.

\begin{de}
Let $k \geq 0$. The map $d_C : \pt_k \times \pt_k \rightarrow \R_+ $ defined by
\[ d_C(t_1 \bullet,t_2 \bullet) = \mathds{1}_{p(t_1 \bullet) \neq p(t_2 \bullet)} + \mathds{1}_{p(t_1 \bullet) = p(t_2 \bullet)} \left( \left( \max_{u_1 \in p(t_1 \bullet), u_2 \in p(t_2 \bullet)} d(\lambda(u_1), \lambda(u_2)) \right) \wedge 1 \right), \]
where $\lambda(u)$ is the label of a node $u$ and $d$ represents any distance on the set of labels (seen as a subspace of $\R^i$ for some $i$), is a distance on $\pt_k$ (for usual reasons). We call it the \emph{coordinate-label distance}.
\end{de}

We define in an analogous manner a distance $d_F$ on $\ft_k$. The spaces $\pt_k$ and $\ft_k$ for $k \geq 0$ are equipped with the corresponding Borel $\sigma$-algebras. We can now state the main theorem of this section.

\begin{thm}\label{bij_drawing_markedtree}
Let $n \geq 0$. Equip the set $\pt_n$ with the coordinate-label distance $d_C$ and $\D_n$ with the Hausdorff distance $d_H$. Then there exists a homeomorphism
\[ \Psi_n : \pt_n \rightarrow \D_n \]
\[ \hspace{1cm} t \bullet \mapsto m ,\]
such that:
\begin{enumerate}
\item Each internal node $u$ of $t \bullet$ corresponds bijectively to an internal vertex $M$ of $m$. Moreover, if $ \lambda(u) = (\C(u),P(u))$ then the barycentric coordinates of the vertex $M$ with respect to $(A,B,C)$ are given by $\C(u).P(u)$.
\item Each leaf $l$ of $t \bullet$ corresponds bijectively to a face $f$ of $m$. Moreover, if $\C(l) = (\C_1,\C_2,\C_3)$ is the label at $l$ and the face $f$ is defined by the three vertices $(A_f,B_f,C_f)$ then $\C(A_f) = \C_1$, $\C(B_f) = \C_2$, $\C(C_f) = \C_3$.
\end{enumerate}
\end{thm}

\begin{rem}
Note that the spaces which are in one-to-one correspondence are both infinite, so that it is not the existence of the bijection as such which is of interest, but the fact that via this bijection all relevant information on a compact triangulation can be read in a coordinate-labelled tree. The measurability of the bijection will allow us to transport distributions.
\end{rem}

\begin{de}\label{def_triangles}
For a node $u$ in a coordinate-labelled tree $t \bullet \in \pt$, we define $T(u)$ to be the triangle whose three points are given by the triplet of coordinates $\C(u)$, and $\tilde{T}(u)$ for the filled triangle. This is illustrated in Figure \ref{index_trg}.
\end{de}

\begin{figure}[!h]
\centerline{\includegraphics{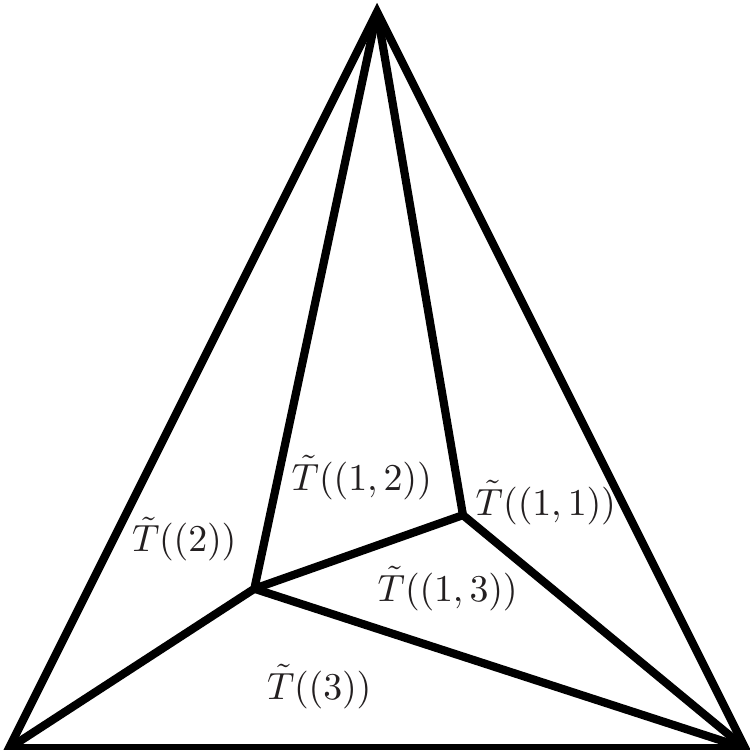}}
\caption{The indexation of triangles}
\label{index_trg}
\end{figure}

\subsection{Proof of Theorem \ref{bij_drawing_markedtree}}

We proceed by induction on $n$. We follow a similar path to the proof of Proposition 1 in \cite{AM}, by constructing the bijection iteratively. For $n=0$ there is no work to do. We have $\D_0= \{ T \}$, $ \pt_0 = \{ \{ \emptyset \}, \C(\emptyset) \}$. By property (c) of Definition \ref{def_pltrees} the coordinates $\C(\emptyset)$ satisfy part 2 of the theorem, as desired.

Now assume we have constructed $\Phi_n$ as in the statement of Theorem \ref{bij_drawing_markedtree}, for some $n \geq 0$. Let $ t \bullet \in \pt_{n+1} $. Denote $t = p( t \bullet)$ and choose a node $u \in t$ such that $u(1),u(2),u(3)$ are leaves of $t$. Now define $t' := t \setminus \{u(1),u(2),u(3),\}$, and $t' \bullet$ to be the coordinate-labelled tree such that its labels coincide with those of $t \bullet$ except at $u$, and where we remove the splitting triplet $P(u) = (P_1,P_2,P_3)$, as $u$ is now a leaf of $t'$. Thus $t' \bullet \in \pt_n$ and by induction we can define $m' := \Psi_n(t') \in \D_n$. Let $f$ be the face of $m'$ corresponding to the leaf $u$ via $\Psi_n$. Write as in the statement of the theorem $(A_f,B_f,C_f)$ for the three vertices defining $f$.

Now let $M$ be the point in $f$ whose barycentric coordinates with respect to $(A_f,B_f,C_f)$ are $(P_1,P_2,P_3)$, and define $m = \Psi_{n+1}(t \bullet)= m' \cup [A_f,M] \cup [B_f,M] \cup [C_f,M]$. It follows that the barycentric coordinates of $M$ with respect to $(A,B,C)$ are $P_1 \C(A_f) + P_2 \C(B_f) + P_3 \C(C_f) = P(u).\C(u)$ by property 2 of the induction hypothesis applied to $u$ in $t'$. Thus, by mapping $u$ to $M$ and all other internal nodes of $ t \bullet $ to their corresponding internal vertex via $\Psi_n$, we see that $\Psi_{n+1}$ satisfies condition 1 of Theorem \ref{bij_drawing_markedtree}. To satisfy condition 2, we map all the leaves of $t'$ except $u$ to their corresponding faces via $\Psi_n$, noting that these faces are untouched by $\Psi_{n+1}$ so that the condition remains satisfied. Finally, we map the leaves $u(1),u(2),u(3)$ respectively to the faces $MB_fC_f,A_fMC_f,A_fB_fM$ of $m$. Because of the local growth property (d) of Definition \ref{def_pltrees}(2), we see that condition 2 is satisfied for these leaves. This iterative construction is illustrated in Figure \ref{loc_bij}.

\begin{figure}[!h]
\centerline{\includegraphics{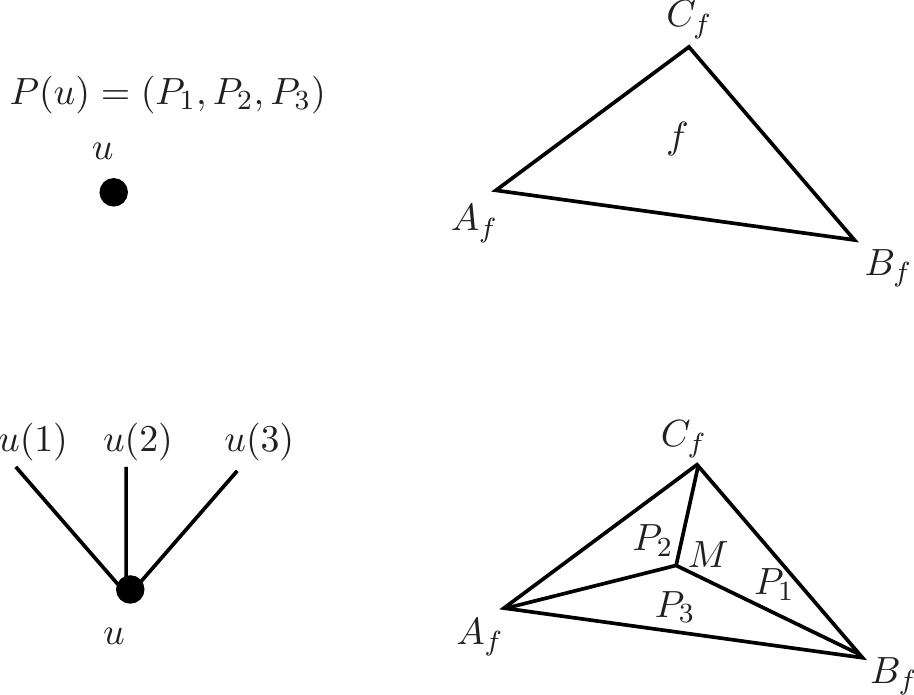}}
\caption{Illustrating the growth property of the bijection $\Psi_n$}
\label{loc_bij}
\end{figure}

\medskip

Two points remain. Firstly, that $\Psi_{n+1}$ is a bijection. But this follows from our construction and the definition of $\D_n$. Indeed $\D_{n+1}$ is obtained from $\D_n$ through the insertion of a vertex anywhere in a given face of an element of $\D_n$, while we have a similar iterative structure for coordinate-labelled trees. It is important here that each vertex is inserted in the interior of some face, and not on it's boundary (since the splitting triplets are in $V_2^*$ and not just in $V_2$), so that the face it is inserted in is defined non ambiguously.

The final point is to prove that $\Psi_{n}$ is bicontinuous with respect to the given distances. For this, we fix some $m \in \D_n$ and $\varepsilon > 0$. Write $t\bullet = \Psi_n^{-1}(m)$. Now there exists $\eta > 0$ such that for any $(\C=(\C_1,C_2,C_3),P=(P_1,P_2,P_3)),(\C'=(\C_1',C_2',C_3'),P'=(P_1',P_2',P_3')) \in V_2^3 \times V_2^*$, if $\|(\C,P) - (\C',P')\| < \eta$ then $\| \C.P - \C'.P' \| < \varepsilon$. We may suppose that $\eta < 1$, so that if $d_C(t \bullet ', t \bullet) < \eta$ we have $p(t \bullet ') = p( t \bullet)$. This implies that if $d_C(t \bullet ', t \bullet) < \eta$ then for any vertex $u \in p( t \bullet)$ the corresponding vertex in $m$ is at distance less than $\varepsilon$ from the corresponding vertex in $m' := \Psi_n(t \bullet ')$. As a consequence, we get that $d_H(m',m) < \varepsilon$ and the continuity of $\Psi_n$ is proved. The bicontinuity stems immediately from the fact that it is a mapping between compact spaces.\qed

\subsection{Introducing randomness}

So far, we have worked in a purely deterministic setting. In this paragraph, we formally introduce the two probability distributions on $D_k$ which will be of interest to us.

\begin{de}
A \emph{splitting law} $\nu$ is a distribution on $\R_+^3$ such that if $P=(P_1,P_2,P_3)$ is distributed according to $\nu$, then:
\begin{enumerate}
\item For any permutation $\sigma$ on $\lbrace 1,2,3 \rbrace$, $(P_{\sigma(1)},P_{\sigma(2)},P_{\sigma(3)})$ has same distribution as $(P_1,P_2,P_3)$, that is the law of $\nu$ is symmetric.
\item For any $i \in \lbrace 1,2,3 \rbrace$, $P_i >0$ a.s..
\item $P_1 + P_2 + P_3 = 1$ a.s..
\end{enumerate}
We denote $\ms$ the set of splitting laws, and say that a random variable $P=(P_1,P_2,P_3)$ is a \emph{splitting ratio} if its distribution is a splitting law.
\end{de}

Fix some $n \geq 0$. We define two probability distributions on $\T_n$.
\begin{itemize}
\item The first, which we denote $\mathbb{U}_n^{\T}$, is the uniform distribution on $\T_n$.
\item The second, which we denote $\mathbb{H}_n^{\T}$, is defined by induction. For $n = 0$, the distribution $\mathbb{H}_0^{\T}$ takes value the unique tree reduced to its root $\{ \emptyset \}$ a.s.. Now suppose we have defined a distribution $\mathbb{H}_{n-1}^{\T}$ on $\T_{n-1}$. Choose $t \in \T_{n-1}$ according to $\mathbb{H}_{n-1}^{\T}$. Conditionally to $t$, choose one of its $2n -1$ leaves uniformly at random, and replace that leaf by an internal node with three children. This gives us a probability distribution $\mathbb{H}_n^{\T}$ on $\T_n$. Note that the weight of a tree is proportional to the number of histories leading to its construction (starting from a single root node).
\end{itemize}

We say that a random variable $t \in \T$ is an \textit{increasing tree} if it has distribution $\mathbb{H}_n^{\T}$ for some $n \geq 0$.

\begin{de}
Let $\nu \in \ms$ be a splitting law, and $(P(u))_{u \in \T^{\infty}}$ be an i.i.d. sequence of random variables with law $\nu$. Let $n \geq 0$.
\begin{enumerate}
\item We denote $\mathbb{U}_n^{\T,\nu}$ (resp. $\mathbb{H}_n^{\T,\nu}$) the distribution of $t^P_n \bullet := \Phi_n (t_n,(P(u))_{u \in t_n^0}) $ where $t_n \in T_n$ is independent from $(P(u))_{u \in \T^{\infty}}$ and has distribution $\mathbb{U}_n^{\T}$ (resp. $\mathbb{H}_n^{\T}$), and $\Phi_n$ is as in Remark \ref{def_Phi}.
\item We define the distributions $\up_n$ and $\qp_n$ to be the respective images of the distributions $\mathbb{U}_n^{\T,\nu}$ and $\mathbb{H}_n^{\T,\nu}$ via the bijection $\Psi_n$ of Theorem \ref{bij_drawing_markedtree}. These are therefore two probability distributions on $\D_n$.
\end{enumerate}
\end{de}

\section{The uniform model}\label{unif}

In this section, we study the asymptotic behaviour of the distribution $\up_n$. That is, we look at random compact triangulations where the underlying stack triangulation is chosen uniformly, and the insertion of vertices is done according to some splitting law $\nu \in \ms$, independent from the choice of the underlying triangulation. We study both the occupation measure, and the asymptotic behaviour of the distribution itself.

\subsection{The occupation measure}\label{sectionUOM}

\begin{thm}\label{UOM}
Let $(m_n)_n$ be a sequence of random compact triangulations, where $m_n$ has distribution $\up_n$. Recall (Definition \ref{def_V}) that $\V(m_n)$ denotes the set of internal vertices of $m_n$. For every $n$, conditionally to $ m_n $, let $U_n$ be a vertex of $\V(m_n)$, chosen uniformly at random. Finally, let $\mu_n$ be the occupation measure of $m_n$, as in \eref{def_OM}. Then
\begin{enumerate}
\item The random point $U_n$ converges in distribution to some random limit point $U_{\infty}$ as $n$ tends to infinity.
\item We have 
\[ \mu_n \cvd \delta_{U_{\infty}} \quad \mbox{as } n \rightarrow \infty,\]
where the convergence is in distribution on the space of probability measures on the filled triangle $\tilde{T}$.
\end{enumerate}
\end{thm}

Theorem \ref{UOM} says is that in the uniform model, all the vertices of the compact triangulation are at the same place, except for a portion which tends to $0$. Although point 2 is stronger than point 1, we state both here as point 1 will be heavily used in the proof of point 2. In Figure \ref{simu1} is a simulation of the vertices of the map $m_n$ where we take for $\nu$ the special case $\nu =  \delta_{\left(\frac13,\frac13,\frac13\right)}$, and $n \sim 10 000$. We can see that the vertices are indeed concentrated at one place.

\begin{figure}[!h]
\centerline{\includegraphics{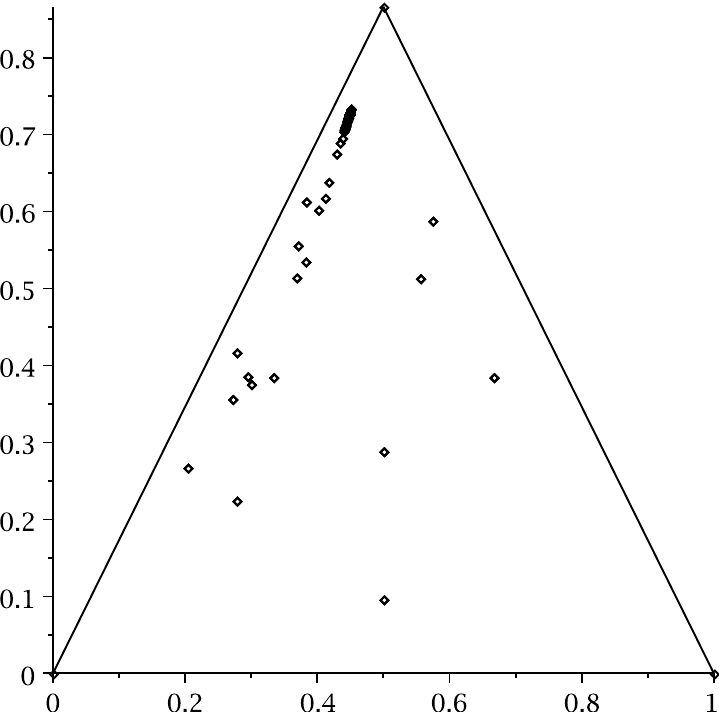}}
\caption{A simulation of the set $\V(m_n)$ where $m_n$ has distribution $\up_n$ and $n \sim 10 000$}
\label{simu1}
\end{figure}

\subsubsection{Proof of Theorem \ref{UOM}.(1)}

We begin by recalling an elementary fact about uniform ternary trees.

\begin{fact}\label{unif_point}
Take $U_n$ as in the statement of Theorem \ref{UOM}, and write $U'_n$ for the corresponding node in the coordinate-labelled tree $t\bullet_n := \Psi_n^{-1}(m_n)$, as in Theorem \ref{bij_drawing_markedtree}. Write $U'_n=(u_1,u_2,\cdots,u_h)$ where $h$ is the height of $U'_n$. Then conditionally to $h$, the random variables $u_1,u_2,...,u_h$ are i.i.d, and are uniformly distributed on $\lbrace 1,2,3 \rbrace$.
\end{fact}

\begin{proof}
By construction of the law $\up_n$, the tree $t_n := p(t\bullet_n)$ follows the uniform distribution on $\T_n$. We now use the following argument. If $t$ is a random ternary tree, chosen uniformly among trees of a given size, then conditionally to their sizes the subtrees at the root $\theta_{(1)}(t),\theta_{(2)}(t),\theta_{(3)}(t)$ are independent, and also follow the uniform distribution. It immediately follows that the $(u_i)$ are i.i.d, and the fact that the law of $u_1$ is uniform on $\lbrace 1,2,3 \rbrace$ stems from the symmetric nature of the uniform distribution in $\T_n$.
\end{proof}

By definition of a coordinate-labelled tree we have the following: let $u$ be an internal node in a coordinate-labelled tree $t$, with coordinates $\C(u)=(\C_1,\C_2,\C_3)$ and splitting triplet $P(u)=(P_1,P_2,P_3)$, then:
\[ \forall i \in \lbrace 1,2,3 \rbrace, \quad \C(u(i))^T=M_{\nu}^{(i)}.\C(u)^T,\]
where $M_{\nu}^{(i)}$ is the three-by-three identity matrix in which the $i$-th line is replaced by $P(u)$, i.e.
\beq\label{defM}
M_{\nu}^{(1)} = \left( \begin{array}{ccc}
P_1 & P_2 & P_3 \\
0 & 1 & 0 \\
0 & 0 & 1
\end{array} \right), \quad
M_{\nu}^{(2)} = \left( \begin{array}{ccc}
1 & 0 & 0 \\
P_1 & P_2 & P_3 \\
0 & 0 & 1
\end{array} \right), \quad
M_{\nu}^{(3)} = \left( \begin{array}{ccc}
1 & 0 & 0 \\
0 & 1 & 0 \\
P_1 & P_2 & P_3
\end{array} \right).
\eq
Henceforth, we will leave out the subscript $\nu$ wherever there is no risk of confusion. Combining this and Fact \ref{unif_point} gives us the following result.

\begin{pro}
Let $m_n,U_n$ be as in the statement of Theorem \ref{UOM}. Write $U'_n$ for the corresponding node in the coordinate-labelled tree $t\bullet_n := \Psi_n^{-1}(m_n)$, and $\C(U'_n)=(\C_1(U'_n),\C_2(U'_n),\C_3(U'_n))$ for the coordinates of $U'_n$. Then for $i \in \{ 1,2,3 \}$, conditionally to $h$ the height of $U'_n$, the law of $\C_i(U'_n)$ is given by the $i$-th row of the product $M_h \cdots M_1$ where the $M_j$ are i.i.d random variables with law $\frac{1}{3} \delta_{M^{(1)}} + \frac{1}{3} \delta_{M^{(2)}} + \frac{1}{3} \delta_{M^{(3)}}$ (the $M^{(k)}$ being defined as in \eref{defM}).
\end{pro}

Now to get the desired convergence of $U_n$, it is of course sufficient to show the convergence of the sequence of coordinates $(\C_n)_{n\geq0}$ where $\C_n$ is the barycentric coordinates of the point $U_n$. By Theorem \ref{bij_drawing_markedtree}(1) the law of $\C_n$ is $P_1\C_1(U'_n)+P_2\C_2(U'_n)+P_3\C_3(U'_n)$ where $P=(P_1,P_2,P_3)$ is a splitting ratio with distribution $\nu$, independent from $\C(U'_n)$. 
The previous proposition gives us the law of $\C(U'_n)$. Moreover, for any $A>0$, $\p( \vert U'_n \vert \geq A )$ tends to zero as $n$ goes to infinity. Thus, to prove Theorem \ref{UOM}, it is sufficient to show the following.

\begin{pro}\label{mat_prod}
Let $(M_i)_{i\geq 1}$ be i.i.d. random variables with law $\frac{1}{3} \delta_{M^{(1)}} + \frac{1}{3} \delta_{M^{(2)}} + \frac{1}{3} \delta_{M^{(3)}}$. Then the product $S_n:=M_n \cdots M_1$ converges a.s. as $n \rightarrow \infty$ to some random matrix $S$ whose three lines are identical.
\end{pro}

\begin{proof}
We write $S_n= \left( \begin{array}{c}
L^{(1)}_n \\
L^{(2)}_n \\
L^{(3)}_n 
\end{array} \right) $. We wish to show that there exists $L_{\infty}$ such that a.s. for all $i \in \{1,2,3 \}$, $L^{(i)}_n \rightarrow L_{\infty}$.

\begin{lem}\label{subseq}
Let $(n_k)$ be some sub-sequence of integers such that $L^{(i)}_{n_k} \rightarrow L^{(i)}$ a.s. for all $i \in \{1,2,3 \}$. Then $L^{(1)}=L^{(2)}=L^{(3)}$ a.s..
\end{lem}

\begin{proof}
We proceed by contradiction. To simplify notation we assume that $L^{(i)}_n \rightarrow L^{(i)}$ a.s. for $i \in \{1,2,3 \}$. Write $A,B,C$ for the three points whose respective coordinates are $L^{(1)},L^{(2)},L^{(3)}$. Similarly write $A_n,B_n,C_n$ for the three points with respective coordinates $L^{(1)}_n,L^{(2)}_n,L^{(3)}_n$. We may assume that $\p(A \neq C) > 0$, and from now on work conditionally to this event. 

Fix some $\varepsilon > 0$ such that $6 \varepsilon < d(A,C) $. Now there exists $N$ such that for any $n \geq N$, we have $d(X_n,X) < \varepsilon$ for $X \in \lbrace A,C \rbrace$. Thus by construction the balls $B(A_n, 2 \varepsilon)$ and $B(C_n, 2 \varepsilon)$ do not intersect, and $B(X, \varepsilon) \subseteq B(X_n, 2 \varepsilon)$ for $X \in \lbrace A,C \rbrace$. Define $Y_n := P_1 A_n + P_2 B_n + P_3 C_n$, where $P=(P_1,P_2,P_3)$ is a splitting ratio, independent from $(A_n, B_n, C_n)$. Then $d(Y_n,A_n) \geq 2\varepsilon$ or $d(Y_n,C_n) \geq 2\varepsilon$, so that $d(Y_n,A) \geq \varepsilon$ or $d(Y_n,C) \geq \varepsilon$. See Figure \ref{smallD} below for an illustration of this situation. Using the definition of the matrices $(M_i)$ we get that with probability equal to $1$ (still conditionally on the event $A \neq C$) there exists $n_0 \geq N$ such that one of the following occurs:
\begin{enumerate}
\item We have $d(Y_{n_0},A) \geq \varepsilon$ and $M_{n_0 + 1} = M^{(1)}$ so that $A_{n_0 +1} = Y_{n_0}$, which contradicts $d(A_{n_0 + 1},A) < \varepsilon$.
\item We have $d(Y_{n_0},C) \geq \varepsilon$ and $M_{n_0 + 1} = M^{(3)}$ so that $C_{n_0 +1} = Y_{n_0}$, which contradicts $d(C_{n_0 + 1},C) < \varepsilon$.
\end{enumerate}
This completes the proof of the lemma.

\begin{figure}[!h]
\centerline{\includegraphics{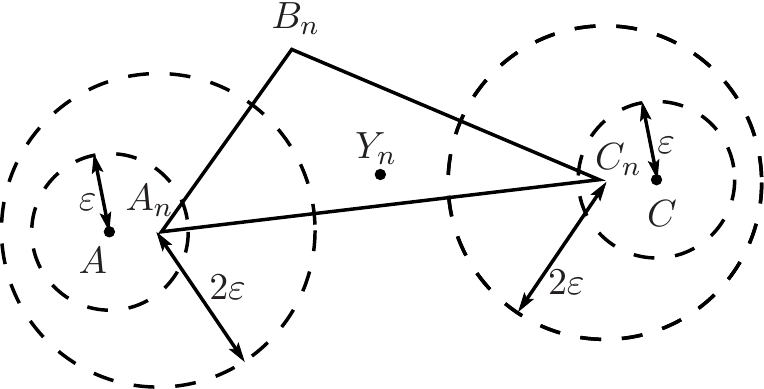}}
\caption{An illustration of the proof of Lemma \ref{subseq}}
\label{smallD}
\end{figure}

\end{proof}

Let us now prove Proposition \ref{mat_prod}. Let $(L,L,L)$ be the a.s. limit along some subsequence of $\big((L^{(1)}_n,L^{(2)}_n,L^{(3)}_n)\big)$. Write $M$ for the point in $\tilde{T}$ with (barycentric) coordinates $L$. Similarly, write $(M^{(1)}_n,M^{(2)}_n,M^{(3)}_n)$ for the points with respective coordinates $(L^{(1)}_n,L^{(2)}_n,L^{(3)}_n)$. Now a.s. for any $\varepsilon > 0$ there exists $N$ such that $d(M^{(i)}_N,M) < \varepsilon$ for all $i \in \{ 1,2,3 \}$. But for $n \geq N$ the points $M^{(i)}_n$ are all in the filled triangle defined by $(M^{(1)}_N,M^{(2)}_N,M^{(3)}_N)$ by construction. It follows therefore that for any $n \geq N$, we have $d(M^{(i)}_n,M) < \varepsilon$ for all $i \in \{ 1,2,3 \}$. This proves that a.s. 
\[ \left(M^{(1)}_n,M^{(2)}_n,M^{(3)}_n\right) \longrightarrow (M,M,M) \quad \mbox{as } n \rightarrow \infty, \]
which is the desired result.
\end{proof}

\subsubsection{Proof of Theorem \ref{UOM}.(2)}

The idea of the proof is as follows. Consider a uniform ternary tree $t_n \in \T_n$ and choose two independent nodes $u^{(1)}_n,u^{(2)}_n$ uniformly at random in $t_n$. Then the greatest common ancestor of these $v_n:=u^{(1)}_n \wedge u^{(2)}_n$ is at height of order $n^{\frac{1}{2}}$. This says that the corresponding two vertices $U^{(1)}_n,U^{(2)}_n$ are in a small triangle. Intuitively, this suggests they will asymptotically be near to each other. This is made clear in the following lemma.

\begin{lem}\label{2points}
Keeping the same notation as in the statement of Theorem \ref{UOM}, conditionally to $m_n$, choose two vertices $U^{(1)}_n,U^{(2)}_n \in \V(m_n)$ independently, uniformly at random. Then the following convergence holds in probability:
\[ \| U^{(2)}_n - U^{(1)}_n \| \cvp 0, \quad \mbox{as } n \rightarrow \infty.\]
\end{lem}

\begin{proof}
Let as before $U'^{(1)}_n$ (resp. $U'^{(2)}_n$) be the node corresponding to $U^{(1)}_n$ (resp. $U^{(2)}_n$) in the tree $t\bullet_n$, via the bijection $\Psi_n$ established in Theorem \ref{bij_drawing_markedtree}. Write $V_n := U'^{(1)}_n \wedge U'^{(2)}_n$ for the greatest common ancestor of these two nodes. It is clear that $ \| U^{(2)}_n - U^{(1)}_n \| \leq \mathrm{diam}(\tilde{T}(V_n))$, where $\mathrm{diam}(S)$ is the diameter of set $S$. Moreover, we know that for any $A>0$, $\p( \vert V_n \vert \geq A )$ tends to zero as $n$ goes to infinity. Using this, and Proposition \ref{unif_point}, it is sufficient to show the following.

\begin{lem}\label{small_diam}
Let $(u_k)_{k \geq 1}$ be a sequence of i.i.d. random variables, uniform on $\lbrace 1,2,3 \rbrace$. Write $W_n := u_1 \cdots u_n \in \T^{\infty}$. Then the following convergence holds in probability:
\[ \mathrm{diam}(\tilde{T}(W_n)) \cvp 0 \quad \mbox{as } n \rightarrow \infty.\]
\end{lem}

In fact, the convergence holds a.s.. Note that the sequence of triangles $ \big( \tilde{T}(W_n) \big)_n $ is non increasing for inclusion, therefore $ \big( 
\mbox{diam}(\tilde{T}(W_n)) \big)_n$ is non increasing, so converges a.s. to some limit $l \geq 0$. Now take some subsequence $(n_k)$ such that the triangle $ \tilde{T}(W_{n_k}) $ converges to some limit triangle $\tilde{T_0} = (A_0,B_0,C_0)$\footnote{When we say ``the triangle converges" here, we mean that the triplet of points of the triangle converges}. We can then use the same proof as for Lemma \ref{subseq} to show that $A_0=B_0=C_0$ a.s., and hence $l=0$ a.s. as desired.
\end{proof}

We now use Lemma \ref{2points} to prove Theorem \ref{UOM}.(2). We denote, for any measure $\mu$ on the triangle $\tilde{T}$ and any measurable function $f$ on $\tilde{T}$, 
\[ \left\langle f, \mu \right\rangle := \int_{\tilde{T}} f \, d\mu .\]
 We show that for any real-valued function $f$ continuous on $\tilde{T}$, $ \left\langle f,\mu_n \right\rangle \, \cvd \, \left\langle f,\delta_{U_{\infty}}\right\rangle = f(U_{\infty})$. It suffices to show that 
\[ \Big( \left\langle f,\mu_n\right\rangle \, , f(U_n) \Big) \cvd \Big(f(U_{\infty}),f(U_{\infty})\Big), \]
where $U_n$ is as in the statement of Theorem \ref{UOM}. Since point $(1)$ of Theorem \ref{UOM} implies that $f(U_n) \cvd f(U_{\infty})$, it suffices to show that 
\[ \left\langle f,\mu_n\right\rangle \, - \, f(U_n) \cvd 0.\]
Now $\E ( \left\langle f,\mu_n\right\rangle ) = \E \big( \frac{1}{n} \sum_{x \in \V(m_n)} f(x) \big) = \E \big(f(U_n)\big)$, thus it is sufficient to show that 
\[\Var \big( \left\langle f,\mu_n\right\rangle \, - \, f(U_n) \big) \longrightarrow 0. \] 
Let $U^{(1)}_n,U^{(2)}_n$ be as in the statement of Lemma \ref{2points}. We have
\[ \Var \big( \left\langle f,\mu_n\right\rangle \, - \, f(U_n) \big) = \E\big( ( \left\langle f,\mu_n\right\rangle \, - f(U_n) )^2 \big)\]
\[= \E\big(f\left(U^{(1)}_n\right)^2 - f\left(U^{(1)}_n\right)f\left(U^{(2)}_n\right)\big), \]
so that
\[ \Var \big( \left\langle f,\mu_n\right\rangle \, - \, f(U_n) \big) \leq \| f \| _{\infty} \E\big( \vert f\left(U^{(1)}_n\right) - f\left(U^{(2)}_n \right) \vert \big) . \]
Since $\tilde{T}$ is compact and $f$ continuous, $f$ is uniformly continuous. Fix some $\varepsilon > 0$. There exists $\eta > 0$ such that for any $x,y \in \tilde{T}$ with $\| x-y \| \leq \eta$, we have $\vert f(x) - f(y) \vert \leq \varepsilon$. Then
\[ \Var \big( \left\langle f,\mu_n\right\rangle \, - \, f(U_n) \big) \leq \| f \| _{\infty}\Big(\varepsilon + 2 \| f \| _{\infty} \p\big(\| U^{(2)}_n - U^{(1)}_n \| > \eta \big) \Big). \]
Using Lemma \ref{2points} we get that
\[ \limsup_{n \rightarrow \infty} \Var \big( \left\langle f,\mu_n\right\rangle \, - \, f(U_n) \big) \leq \varepsilon \| f \| _{\infty}, \]
and since this holds for any $\varepsilon > 0$, the desired result follows. This completes the proof of Theorem \ref{UOM}. \qed

\medskip

It may also be interesting to obtain information on the law of the limit point $U_{\infty}$, since this point is where the occupation measure is concentrated asymptotically. Proposition \ref{mat_prod} tells us that the coordinates $C_{\infty}$ of the limit point $U_{\infty}$ follow the law of one line of this matrix $S$, and satisfies the following equation in distribution:
\beq\label{dist_eq}
C_{\infty} \eqd C_{\infty}.M_{\nu} \mbox{, where $M_{\nu}$ has distribution }\frac{1}{3} \delta_{M_{\nu}^{(1)}} + \frac{1}{3} \delta_{M_{\nu}^{(2)}} + \frac{1}{3} \delta_{M_{\nu}^{(3)}}, 
\eq
and the $M_{\nu}^{(i)}$ are defined as in \eref{defM}.
This can be interpreted as follows. Split the original triangle $T$ in three using the splitting law $\nu$, and pick one of the three subsequent triangles uniformly at random. Now choosing a point with respect to $C_{\infty}$ in that triangle is the same (has the same law) as choosing a point with respect to $C_{\infty}$ in $T$. The distribution of $C_{\infty}$ is thus the limit distribution of a (very) simple Markov chain.

\begin{pro}
Let $\mathcal{M}_2(C)$ be the set of symmetric (probability) laws on $V_2^*$. For any splitting law $\nu \in \ms$, the distribution equation $C_{\infty} \eqd C_{\infty}.M_{\nu}$ has a unique solution $C_{\infty} \in \mathcal{M}_2(C)$.
\end{pro}

This tells us that Equation \eref{dist_eq} characterises the distribution of the limit point.

\begin{proof}
We endow $\mathcal{M}_2(C)$ with the usual $\mathcal{L}^2$ norm denoted $\|.\|_2$. This makes it complete, and thus by Banach's fixed point theorem it is sufficient to show that the map $ \left\lbrace \begin{array}{c}
\mathcal{M}_2(C) \rightarrow \mathcal{M}_2(C) \\
\hspace{0.7cm} \mathbf{L}(X) \mapsto \mathbf{L}(X.M_{\nu}) \end{array} \right.$, where $\mathbf{L}(Y)$ denotes the law of a random variable $Y$, is a contraction.

\smallskip

Take $\mu \in \mathcal{M}_2(C)$ and let $X=(X_1,X_2,X_3)$ have distribution $\mu$. Write for short $m_2= \E(X_i^2)$ and $m_{1,1}=\E(X_iX_j)$ for $i \neq j$. Then
\[ \| X \|_2 ^2 = 3m_2. \]
We now wish to compute $\| XM_{\nu} \|_2 ^2$. We have
\[\| XM_{\nu} \|_2 ^2 = \frac{1}{3} \E \Big( \sum_{i=1}^3 XM^{(i)}(M^{(i)}) ^TX^T \Big). \]
Now a computation gives us that
\[ M^{(1)}(M^{(1)})^T = \left( \begin{array}{ccc}
\vert P \vert ^2 & P_2 & P_3 \\
P_2 & 1 & 0 \\
P_3 & 0 & 1
\end{array} \right), \]
where $\vert P \vert ^2 := P_1^2 + P_2 ^2 + P_3^2 $. Thus
\[ XM^{(1)}(M^{(1)})^TX^T = X_1^2 \vert P \vert ^2 + X_2^2 P_2 + X_3^2P_3 + 2(X_1X_2P_2 + X_1X_3P_3).\]
It follows that 
\[\E\Big(XM^{(1)}(M^{(1)})^TX^T\Big) = m_2\Big( \E \big(\vert P \vert ^2 \big) + \frac{2}{3} \Big) + \frac{4}{3} m_{1,1}. \]
Here we use the symmetry of $\nu$ (hence in particular $\E(P_i) = \frac{1}{3})$. Since the above equality is symmetric, it immediately follows that
\beq\label{norm}
\| XM_{\nu} \|_2 ^2 = m_2\left( \E \big(\vert P \vert ^2 \big) + \frac{2}{3} \right) + \frac{4}{3} m_{1,1}.
\eq
Now $\E(P_1^2) \leq \E(P_1) = \frac{1}{3} $ since $P_1 \leq 1$ a.s. and moreover, this inequality is strict since $P_1=0 \mbox{ or } 1$ a.s. is not allowed. Write $a=3\E(P_1^2)=\E \big(\vert P \vert ^2 \big)<1$. By the Cauchy-Schwarz inequality
\[ m_{1,1} = \E(X_1X_2) \leq \sqrt{\E(X_1^2)} \sqrt{\E(X_2^2)} = m_2. \]
It follows from \eref{norm} that 
\[\| XM_{\nu} \|_2 ^2 \leq (a+2)m_2 = \frac{a+2}{3} \| X \|_2 ^2, \]
and since $a<1$ this shows that the map $X \rightarrow XM_{\nu} $ is indeed a contraction.
\end{proof}

\paragraph{Special case:} when $P = (\frac{1}{3},\frac{1}{3},\frac{1}{3})$ a.s., the law of $U_{\infty}$ is the uniform distribution on $\tilde{T}$.

\noindent Indeed, putting a point at the centre of gravity of a triangle, choosing one of the three resulting triangles uniformly at random and placing a point uniformly in that triangle, is the same as placing a point uniformly in the original triangle.

\subsection{The drawing of the triangulation}\label{UD}

The previous results give us information on the asymptotic behaviour of the occupation measure, and thus tell us where the vertices are located asymptotically. In this section we obtain information on the behaviour of the drawings themselves, that is the behaviour of compact triangulations under $\up_n$. We immediately state the main result.

\begin{thm}\label{conv_draw}
Let $\nu \in \ms$ be a splitting law and let $(m_n)_{n \geq 0}$ be a sequence of compact triangulations under the distribution $\up_n$. There exists a random compact space $m_{\infty}$ such that
\[ m_n \longrightarrow m_{\infty}, \quad \mbox{as } n \rightarrow \infty \]
where the convergence holds in distribution in the set of compact subspaces of the filled triangle $\tilde{T}$ equipped with the Hausdorff distance.
\end{thm}

The limit space $m_{\infty}$ is characterised as follows. Start with the initial triangle $T$ split in three by adding a point according to $\nu$. Pick one of these three triangles uniformly at random, call it $T'$. For each of the other two triangles, consider independently a random critical Galton-Watson ternary tree - this object shall be defined later in the paper - and draw the corresponding compact triangulation (each vertex insertion according to $\nu$, independently from all previous insertions and from the trees). Iterate ad infinitum this construction, replacing $T$ with $T'$, and take the closure of the space obtained (so as to have a compact space). Figure \ref{simu2} illustrates this convergence, showing a simulation of the map $m_n$ with $n \sim 10000$. The fact we can only see a handful of ``macroscopic" triangles suggests the convergence of the drawings.

To show Theorem \ref{conv_draw}, we restrict ourselves to the special case where the splitting law 
\beq\label{nu0}
\nu = \nu^0 := \delta_{\left(\frac13,\frac13,\frac13\right)},
\eq 
that is, a splitting ratio with law $\nu^0$ takes value $\left(\frac13,\frac13,\frac13\right)$ a.s.. There is no additional difficulty in the general case, but this special case simplifies certain statements such as Theorem \ref{cont_thm}, as well as certain formulae such as \eref{dist_bar}. We will be careful to always specify how we would proceed in the general case. Recall the definitions of the bijections $\Phi,\Psi$ in Remark \ref{def_Phi} and Theorem \ref{bij_drawing_markedtree}. We define a map
\beq\label{def_psi0}
\begin{array}{c}
\Psi^0 : \T \longrightarrow E \\
\hspace{5.8cm} t \longmapsto \overbar{ \Psi \circ \Phi \left(t,\left(\left(\frac13,\frac13,\frac13\right),u \in t^0 \right)\right)},
\end{array}
\eq
where $E$ is the set of compact subspaces of $\tilde{T}$, and $\bar{S}$ denotes the closure of a subspace $S \subseteq \tilde{T}$. In words, we take a tree $t$, make it a fragmentation labelled tree by adding the labels $\left(\frac13,\frac13,\frac13\right)$ at each internal vertex, and map it to its corresponding compact triangulation via the bijections established in Section \ref{sec_def} (taking the closure if the tree is infinite, so as to always work with compact spaces). Our main tool is the local convergence of Galton-Watson trees, and our main reference \cite{Gil}.

\begin{figure}[!h]
\centerline{\includegraphics{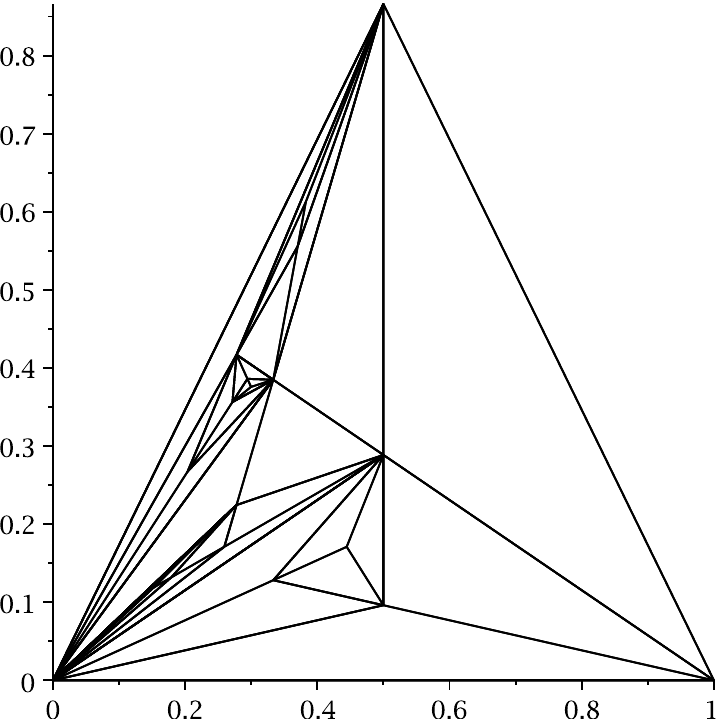}}
\caption{A simulation of a map $m_n$ under the distribution $\mathbb{U}_n^{\nu^0}$, with $n \sim 10000$.}
\label{simu2}
\end{figure}

\subsubsection*{Galton-Watson trees and local convergence}

\begin{de}
A $ \zeta$ Galton-Watson (or GW($\zeta)$-) tree is a random variable $ \tau \in \U$ such that
\begin{enumerate}
\item $ k_{\emptyset}(\tau) $ has law $ \zeta $, i.e. $ \p ( k_{\emptyset}(\tau) = k ) = \zeta (k) $ for any $k \in \N$.
\item For any $k$ such that $ \zeta(k) > 0 $, under the conditional probability $ \p(. \, \big\vert k_{\emptyset}(\tau) = k ) $, the trees $ \theta_{(1)}( \tau ),\theta_{(2)}( \tau ),\cdots,\theta_{(k)}( \tau ) $ are i.i.d and have the same law as $\tau$ under $\p$.
\end{enumerate}
\end{de}

\begin{pro}\label{unif-GW}
Let $\xi$ have law $\frac{2}{3} \delta_0 + \frac{1}{3} \delta_3$, and $\tau$ be a GW($\xi$)-tree. Then a.s. $\tau \in \T$ and for any $n \geq 0$, conditionally to the event $\tau \in \T_n$, $\tau$ is uniform in $\T_n$.
\end{pro}

\begin{proof}
First, $\tau$ is a.s. a ternary tree since by definition of $\xi$ every node has three children or none. Now for any $t \in \T_n$, for some $n\geq0$,
\[ \p(\tau = t) = \left( \frac{1}{3} \right)^n \left( \frac{2}{3} \right)^{2n+1} ,\]
since any $t \in \T_n$ has $n$ internal nodes which each have three children and $2n+1$ leaves. Thus all trees with the same size have the same weight.
\end{proof}

We can therefore view a uniform ternary tree in $\T_n$ as a GW($\xi$)-tree, conditional to have size $n$. We now define the topology of local convergence on trees.

\begin{de}
Let $t \in \U$ be a planar tree, and $r > 0$ some real number. Then $B_r(t)$ is the subtree of $t$ whose vertices all have height at most $r$, that is
\[ B_r(t) = \lbrace u \in t; \ \vert u \vert \leq r \rbrace.\]
\end{de}

\begin{de}\label{loc_conv_def}
Let $t,t' \in \U$ be two planar trees. Define the distance $\tilde{d}$ between $t$ and $t'$ by
\[ \tilde{d}(t,t') = \inf \left\lbrace \frac{1}{r+1}; \ B_r(t) = B_r(t'), \, r \in \R \right\rbrace.\]
\end{de}

One checks that $\tilde{d}$ is indeed a distance. The following proposition is a consequence of Proposition \ref{unif-GW} and Theorem III.3.1 in \cite{Gil}.

\begin{pro}\label{cv_loc}
Let $t_n$ be a uniform tree in $\T_n$. Then there exists a random variable $t_{\infty} \in \T$ such that 
\[ t_n \longrightarrow t_{\infty}, \quad \mbox{as } n \rightarrow \infty, \]
where the convergence holds in distribution in $\T$ equipped with the distance $\tilde{d}$. Moreover, the following properties hold a.s.:
\begin{enumerate}
\item $t_{\infty}$ has a unique infinite branch, written $t_{\infty}^0 = \emptyset,t_{\infty}^1,t_{\infty}^2,\cdots$.
\item For any $k$, conditionally to $t_{\infty}^k$, the law of $t_{\infty}^{k+1}$ is $\frac{1}{3} (\delta_{t_{\infty}^k(1)} + \delta_{t_{\infty}^k(2)} + \delta_{t_{\infty}^k(3)} )$. That is, the infinite branch is an infinite sequence of i.i.d. uniform left, middle, and right turns.
\item For any $u$ on the infinite branch, the two finite subtrees among $\theta_{u(1)}(t_{\infty}), \theta_{u(2)}(t_{\infty}), \theta_{u(3)}(t_{\infty})$ are independent GW($\xi$) trees, where $\xi$ has law $\frac{2}{3} \delta_0 + \frac{1}{3} \delta_3$.
\end{enumerate}
\end{pro}

Now to prove Theorem \ref{conv_draw} in the special case \eref{nu0}, it is sufficient to have some continuity of the function $\Psi^0$ defined by \eref{def_psi0}. In fact, this function is not continuous on $\T$. However, Theorem 25.7 in \cite{Bil} says that to transport convergence in distribution via a function $f$, it is sufficient that $f$ be continuous on the support of the limit in distribution. Therefore, given the properties of $t_{\infty}$ listed in Proposition \ref{cv_loc}, the following suffices.

\begin{thm}\label{cont_thm}
Let $\T$ be equipped with the distance of local convergence $\tilde{d}$. Let $t^0 \in \T$ be a tree with exactly one infinite branch, and assume that along the infinite branch there are an infinity of left, middle, and right turns, i.e. if $(u_0,u_1,\cdots)$ is the infinite branch, then 
\beq\label{cond_tree}
\vert \lbrace i; \ u_i = j \rbrace \vert = \infty \mbox{ for any }j=1,2,3.
\eq 
Then the map $\Psi^0 : \T \longrightarrow E$ defined by \eref{def_psi0} is continuous at $t^0$, where $E$ is equipped with the Hausdorff distance.
\end{thm}

\begin{rem}
In the general case where $\nu \neq \nu^0$ this theorem should be re-stated as a continuity theorem of a function which maps the set of distributions on $\ft$  to the set of distributions on $E$, both sets equipped with the topology of weak convergence. The path of the proof remains unchanged, though formula \eref{dist_bar} is more complicated as is therefore the proof of Lemma \ref{small_infinite_branch}.
\end{rem}

\begin{proof}
Let $t^0$ be as in the statement of the Theorem \ref{cont_thm} and write $(u_0,u_1, \cdots) $ for its infinite branch. Let $(t_n)$ be a sequence of trees in $\T$ such that $t_n \rightarrow t^0$ as $n$ tends to infinity for the distance $\tilde{d}$. Define $m_n = \Psi^0(t_n)$ and $m^0 = \Psi^0(t^0)$. We wish to show that $m_n \rightarrow m^0$ for the Hausdorff distance. Recall from Definition \ref{def_triangles} that for a node $u \in t$, we write $\tilde{T}(u)$ for the corresponding (filled) triangle in the compact triangulation.

\begin{lem}\label{small_infinite_branch}
Let $(u_0,u_1,\cdots)$ be the infinite branch of $t^0$ satisfying condition \eref{cond_tree}. Then
\[ \mathrm{diam}\left(\tilde{T}((u_0,u_1, \cdots ,u_k))\right) \longrightarrow 0, \quad \mbox{as } k \rightarrow \infty. \]
That is, the diameter of the triangle corresponding to the $k$-th node of the infinite branch of $t^0$ tends to zero as $k$ tends to infinity.
\end{lem}

\begin{proof}
Consider Figure \ref{dist_bar1}, where $M$ is the centre of gravity of the triangle. Then a computation gives us
\beq\label{dist_bar}
d = \frac{1}{3} \sqrt{ 2a^2 + 2c^2 - b^2 } \leq \frac{2}{3} \max\{a,b,c\}.
\eq

\begin{figure}[!h]
\centerline{\includegraphics{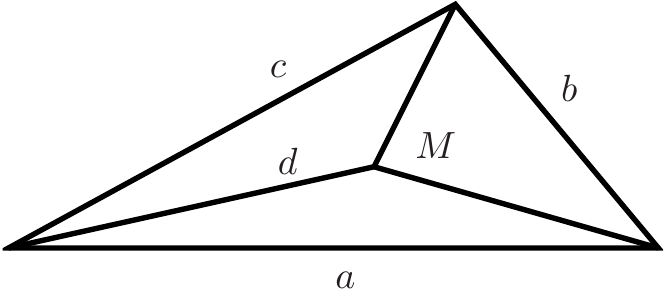}}
\caption{}
\label{dist_bar1}
\end{figure}

It follows that if $k_1 := \min \{ k \geq 0; \, \vert \lbrace i \leq k; \ \forall j \in \{1,2,3\}, \, u_i=j \rbrace \vert \geq 1 \}$ (that is, there is at least one left, right and middle turn in the $k_1$ first steps), then $ \mbox{diam}\left(\tilde{T}((u_0, \cdots, u_{k_1}))\right) \leq \frac{2}{3} \mbox{diam}(\tilde{T}) $. Define inductively, for $l \geq 1$,
\[ k_{l+1} := \min \{ k \geq k_l; \, \forall j \in \{1,2,3\}, \, \vert \lbrace k_l < i \leq k; \  u_i=j \rbrace \vert \geq 1 \}.\]
Condition \eref{cond_tree} implies that for any $l$, $k_l$ is finite. Moreover, we have:
\[ \mbox{diam}\left(\tilde{T}((u_0, \cdots, u_{k_l}))\right) \leq \left( \frac{2}{3} \right)^l \mbox{diam}(\tilde{T}), \]
and taking $l \rightarrow \infty$ completes the proof of Lemma \ref{small_infinite_branch}.
\end{proof}

Now to prove Theorem \ref{cont_thm}, fix some $\varepsilon > 0$, and choose $k$ such that $\mbox{diam}\left(\tilde{T}((u_0, \cdots, u_k))\right) \leq \varepsilon$. Write $u^k=(u_0, \cdots, u_k)$. By the assumptions made on $t^0$ and by definition of the distance $\tilde{d}$, for sufficiently large $n$ the trees $t_n$ and $t^0$ coincide except perhaps on the subtrees $\theta_{u^k}(t_n)$ and $\theta_{u^k}(t^0)$. But this immediately implies that $d_H(m_n,m^0) \leq \mbox{diam}\big(\tilde{T}(u^k)\big) \leq \varepsilon$, and the theorem is proved.
\end{proof}

\section{The increasing case}\label{incr}
In this section we study the asymptotic behaviour of $\qp_n$. That is, at every step, we choose one of the faces of our triangulation uniformly at random and split it into three. We call this the \textit{increasing case}. 

We will see that the asymptotic behaviour of the occupation measure is different to the uniform case. Intuitively this is because, in the uniform case, the distance between two vertices chosen at random tends to zero, since the height of their greatest common ancestor tends to infinity, whereas in the increasing case its law converges to a geometric distribution.

\subsection{The key ingredient: Poisson-Dirichlet fragmentation}\label{frag}

Here we give a construction of the increasing ternary tree as the underlying tree of a fragmentation tree. First, let us describe the deterministic fragmentation tree associated to a sequence of choices $\textbf{u}=(u_i)_{i \geq 1}$ with $u_i \in [0,1)$ for any $i$, and a sequence $\textbf{y}=(y^u)_{u \in \T^{\infty}}$ where for all $u \in \T^{\infty}$, $y^u=(y_1^u,y_2^u,y_3^u)\in V_2^*$.

With these sequences, we associate a sequence $\textbf{F}_n=F(n,\textbf{u},\textbf{y})$ of fragmentation trees with $2n+1$ leaves, each node being marked with a sub-interval of $[0,1)$, as follows.

\begin{itemize}
\item At time 0, $F_0$ is the root tree $\lbrace \emptyset \rbrace$ marked by $I_{\emptyset} = [0,1)$.
\item Assume that at time $k$ the tree $F_k$ is built, and that it is a ternary tree with $2k+1$ leaves, each node $u$ being marked by a semi-open interval $I_u = [a_u,b_u) \subseteq [0,1)$. Moreover, assume that the leaf intervals $(I_l \mbox{, $l$ is a leaf of } F_k)$ form a partition of $[0,1)$. The tree $F_{k+1}$ is then built as follows. Denote $\tilde{l}$ the (unique) leaf of $F_k$ such that $u_{k+1} \in I_{\tilde{l}}$. We give to $\tilde{l}$ three children $\tilde{l}(1),\tilde{l}(2),\tilde{l}(3)$ and mark each of these with a sub-interval of $I_{\tilde{l}}$ whose lengths are prescribed by $y^{\tilde{l}}$. More specifically, if $I_{\tilde{l}} = [a,b)$ then we take $I_{\tilde{l}(1)} = [a, a + (b-a)y^{\tilde{l}}_1)$, $ I_{\tilde{l}(2)} = [a + (b-a)y^{\tilde{l}}_1), a + (b-a)y^{\tilde{l}}_1) + (b-a)y^{\tilde{l}}_2)$, $I_{\tilde{l}(3)} = [a + (b-a)y^{\tilde{l}}_1) + (b-a)y^{\tilde{l}}_2, b)$.
\end{itemize}

Given a fragmentation tree $F$ we will write $\pi(F)$ for the underlying tree (that is, the fragmentation tree with marks removed).
 
\begin{de}
The $2$-dimensional Dirichlet distribution with parameter $\alpha \in (0, +\infty)$, denoted $\Dir_{2}(\alpha)$ is the probability measure on $V_{2}$ with density
\[ f_{\alpha,2}(x_1, x_2, x_3) := \frac{\Gamma(3 \alpha)}{\Gamma(\alpha)^3} \, x_1^{\alpha-1} x_2^{\alpha-1} x_3^{\alpha-1} \]
with respect to the uniform measure on $V_{2}$.
\end{de}

The following fundamental result is due to Albenque and Marckert \cite{AM}.

\begin{thm}\label{frag_tree}
Let $\textbf{U}=(U_i)_{i \geq 1}$ be a sequence of i.i.d. random variables, uniform on $[0,1)$, and $\textbf{Y}=(Y^u)_{u \in \T^{\infty}}$ be a sequence of i.i.d. random variables with $\Dir_2(\frac{1}{2})$ distribution. Now let $\textbf{F}_n = F(n,\textbf{U},\textbf{Y})$ be the sequence of corresponding random fragmentation trees as described above. Then for any $n \geq 0$ the underlying ternary tree $\pi(\textbf{F}_n)$ follows the distribution of an increasing ternary tree on $\T_n$.
\end{thm}

One particular consequence of this result is the following. Let $t_n \in \T_n$ be a family of increasing ternary trees. Then the proportion of internal nodes in each of the first three subtrees $(\mathcal{P}_1,\mathcal{P}_2,\mathcal{P}_3)$, where $\mathcal{P}_i$ is the proportion of internal nodes in the $i$-th first subtree of $t_n$, that is $\mathcal{P}_i := \frac{1}{n} \sharp \lbrace u \in t_n^0; \ u=(i,u_2, \cdots, u_h) \rbrace$, converges in distribution to a $\Dir_2(\frac{1}{2})$ distribution.

\subsection{Convergence of the occupation measure}

In this section we show that the occupation measure $\mu_n$ defined by \eref{def_OM} converges to a random measure $\mu$. Let $\nu$ be a splitting law and $(P(u), u \in \T^{\infty})$ a sequence of i.i.d. splitting ratios with distribution $\nu$. Recall the previous construction. Let $\textbf{U}=(U_i)_{i \geq 1}$ be a sequence of i.i.d. random variables, uniform on $[0,1]$, and $\textbf{Y}=(Y^u)_{u \in \T^{\infty}}$ be a sequence of i.i.d. random variables with $\Dir_2(\frac{1}{2})$ distribution. Let $\textbf{F}_n = F(n,\textbf{U},\textbf{Y})$ be the sequence of corresponding random fragmentation trees, and let $I^u$ be the interval marked at the node $u$. If we write $t_n = \pi(\textbf{F}_n)$ for the underlying ternary tree, then $t_n$ is an increasing tree of size $n$ according to the previous theorem. Let $m_n := \Psi_n \circ \Phi_n \big((t_n,(P(u),u \in t_n^0))\big)$ be the corresponding compact triangulation. By definition, $m_n$ has distribution $\qp_n$. Write $\mu_n$ for its occupation measure as defined by \eref{def_OM}. The remark at the end of Section \ref{frag} says that 
\beq\label{OM_trg_conv}
\forall u \in t^0, \  \mu_n\left(\tilde{T}(u)\right) \rightarrow \vert I^u \vert.
\eq 
This is in fact just the law of large numbers. Indeed, the quantity $\mu_n(\tilde{T}(u))$ is the proportion of uniform random variables on $[0,1]$ which fall in a sub-interval $I^u$. as such, with this construction the convergence in \eref{OM_trg_conv} is a.s..

\begin{de}
Let $u_1,u_2, \cdots ,u_k$ be $k$ nodes of $\T^{\infty}$. We say that $u_1, \cdots ,u_k$ are \textbf{covering} if the following two conditions hold:
\begin{enumerate}
\item We have $ \bigcup_i \tilde{T}(u_i) = \tilde{T}$.
\item For any $i \neq j$, $\mathrm{Int}\left(\tilde{T}(u_i)\right) \cap \mathrm{Int}\left(\tilde{T}(u_j)\right) = \emptyset$, where $\mathrm{Int}(S)$ denotes the interior of a set $S$.
\end{enumerate}
\end{de}

Another important property of the occupation measure $\mu_n$ is that it has weight zero along the edges of the triangles $\tilde{T}(u)$. Indeed, the vertices are always added to the interior of the triangles, so that
\beq\label{border_cond}
\forall u \in \T^{\infty}, \ \mu_n\left( \partial\left( \tilde{T}(u)\right)\right) \rightarrow 0,
\eq
where $\partial S$ represents the boundary of a set $S$. Once again, in our construction this convergence holds a.s..

We now state the main result of this section.

\begin{thm}\label{well_def}
Let $(l^u)_{u \in \T^{\infty}}$ be a sequence of positive random variables such that:
\beq\label{covering_prop}
\mbox{for any nodes } u_1,\cdots ,u_k \in \T^{\infty}, \mbox{ if } u_1\cdots ,u_k \mbox{ are covering, then a.s. } l^{u_1} + \cdots + l^{u_k} = 1.
\eq 
Then there exists an a.s. unique random measure $\mu$ on the triangle $\tilde{T}$ such that the following hold a.s.:
\begin{enumerate}
\item For any $u \in \T^{\infty}$, $\mu\left(\tilde{T}(u)\right) = l^u$.
\item For any $u \in \T^{\infty}$, $\mu\left(\partial \left( \tilde{T}(u)\right)\right) = 0$.
\end{enumerate}
\end{thm}

Since the random variables $\vert I^u \vert$ satisfy condition \eref{covering_prop}, using the convergences of \eref{OM_trg_conv} and \eref{border_cond} we obtain the following consequence.

\begin{thm}\label{IOM}
Let $l^u := \vert I^u \vert$ for all nodes $u \in \T^{\infty}$, and let $\mu$ be the unique measure of Theorem \ref{well_def} for this choice of $(l^u)$. Then the following convergence
\[ \mu_n \cvd \mu, \quad \mbox{as } n \rightarrow \infty \]
holds in distribution in the set of probability measures on $\tilde{T}$ equipped with the topology of weak convergence.
\end{thm}

\begin{rem}
Theorem \ref{well_def} tells us that the information on the triangles $\tilde{T}(u)$ is sufficient to characterise the measure $\mu$. It is crucial that there is no mass on the edges of the triangles here. Indeed, if there were some mass on the edge $[AB]$ of the original triangle, the knowledge of just the values of $\left(\mu(\tilde{T}(u)), u \in \T^{\infty} \right)$ would not be sufficient to obtain information on how that mass is distributed. 
\end{rem}

\begin{proof}
The existence of $\mu$ is a consequence of the property \eref{covering_prop} and Kolmogorov's extension theorem. Let us prove uniqueness. Let $\mu,\mu'$ be two measures on $\tilde{T}$ satisfying the conditions of Theorem \ref{well_def}. For the remainder of the proof, we work with a fixed $\omega$ in our probability space $\Omega$, where the properties (1) and (2) of Theorem \ref{well_def} hold.

Define the set $\hat{T}$ as the triangle $T$ with the boundaries of all the triangles $\tilde{T}(u)$ removed, that is
\[\hat{T} := \tilde{T} \setminus \bigcup_{u \in \T^{\infty}} \partial \left( \tilde{T}(u) \right).\]
Because of property (2) of Theorem \ref{well_def}, we may view $\mu$ and $\mu'$ as measures on $\hat{T}$. Now  the sets $\left(\tilde{T}(u) \cap \hat{T}, u \in \T^{\infty} \right)$ form a basis of open sets for a certain topology on $\hat{T}$, say $\mathcal{O}'$. We first show the following.

\begin{lem}\label{topo}
Let $\mathcal{O}$ denote the topology induced by the usual metric topology on $\hat{T}$. Then $\mathcal{O}' = \mathcal{O}$.
\end{lem}

\begin{proof}
First, note that $\mathcal{O}' \subseteq \mathcal{O}$, since for any $u \in \T^{\infty}$, the set $T(u) \cap \hat{T}$ is an open set for the metric topology on $\hat{T}$. To show the converse, we show that 
\beq\label{topo_eq}
\forall O \in \mathcal{O},\, \exists u \in \T^{\infty},\ T(u) \subseteq O.
\eq
Fix $O \in \mathcal{O}$ and $x \in O$. Define $u^n(x)$ to be the unique vertex $u \in \T^{\infty}$ s.t. $|u| = n$ and $x \in T(u)$. The uniqueness of $u^n(x)$ stems from the fact that $x \notin \bigcup_{u \in \T^{\infty}} \partial \left( \tilde{T}(u) \right)$. For simplicity we write $T^n(x) := T(u^n(x))$. Now to show \eref{topo_eq}, it is sufficient to show that
\[ \mathrm{diam} (T^n(x)) \rightarrow 0, \mbox{ as } n \rightarrow \infty.\]

We write, for any $n$, $u^n(x)=(u_1(x),\cdots,u_n(x))$. Notice that by construction, the $u_i(x)$ are well defined (i.e. they do not depend on $n$). Now if we show that the sequence $(u_i(x),i\geq1)$ satisfies condition \eref{cond_tree}, then we can follow the path of the proof of Lemma \ref{small_infinite_branch} to get the desired result. Let us therefore show that
\beq\label{end_pf} \forall j \in \{1,2,3\}, \  \vert \lbrace i; \ u_i(x) = j \rbrace \vert = \infty.
\eq 

We proceed by contradiction. If \eref{end_pf} doesn't hold, then there are two possibilities:
\begin{enumerate}[(1)]
\item There exists $N \in \N$ and $j \in \{1,2,3\}$, such that for all $i \geq N$, $u_i(x)=j$, that is there is exactly one value of $j$ such that $\vert \lbrace i; \ u_i(x) = j \rbrace \vert$ is infinite.
\item There exists $N \in \N$ and $j \in \{1,2,3\}$, such that for all $i \geq N$, $u_i(x) \neq j$ and for $k \in \{1,2,3\} \setminus j$, $\vert \lbrace i; \ u_i(x) = k \rbrace \vert = \infty$, that is there are exactly two values of $j$ such that $\vert \lbrace i; \ u_i(x) = j \rbrace \vert$ is infinite.
\end{enumerate} 

Consider case (1). Let $N \in \N$ so that for example for all $i \geq N$, $u_i(x)=1$. Write $T_i(x) = (A_i(x),B_i(x),C_i(x))$ for any $i$. Now for any $i \geq N$, we have $B_i(x) = B_N(x) :=B(x) $ and $C_i(x) = C_N(x) := C(x)$ (recalling the ordering of triangles after splitting as shown in Figure \ref{trg_order}). Now if $A(x)$ is a limit point of some subsequence of $(A_n(x))_{n \geq N}$ we can show, using similar arguments as in the proof of Lemma \ref{subseq}, that $A(x) \in [B(x) \, C(x)]$. This implies that as $n$ tends to infinity, the distance between $A_n(x)$ and the line segment $[B(x) \, C(x)]$ tends to zero (see Figure \ref{zero_dist} below). But this would imply that $x \in [B(x) \ C(x)]$, which is impossible since $x \notin \bigcup_{u \in \T^{\infty}} \partial \left( \tilde{T}(u) \right)$. Figure \ref{zero_dist} provides an illustration of this case.

\begin{figure}[!h]
\centerline{\includegraphics{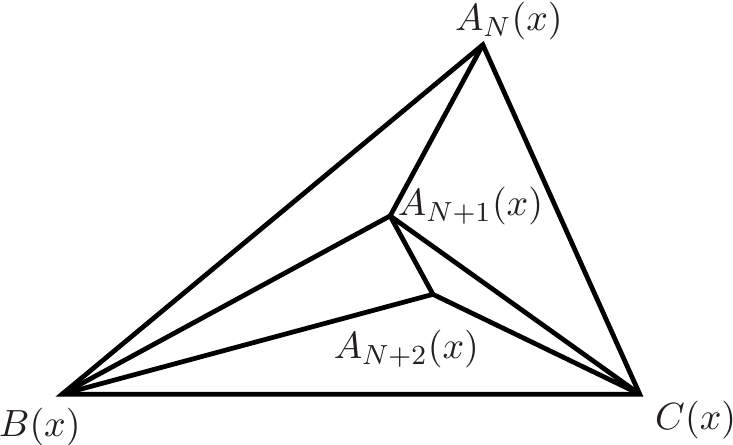}}
\caption{Case (1)}
\label{zero_dist}
\end{figure}

Now consider case (2). We suppose that for $i \geq N$, $u_i(x) \neq 1$ and that $\vert \lbrace i; \ u_i(x) = j \rbrace \vert = \infty$ for $j=1,2$. We still write $T_i(x) = A_i(x),B_i(x),C_i(x))$ for any $i$. Now for any $i \geq N$ we have $A_i(x) = A_N(x) := A(x)$. As above, one shows that the sequences $d(A(x),B_n(x))$, $d(A(x),C_n(x))$ both tend to zero as $n$ tends to infinity, so that we should have $x = A(x)$. But this contradicts once more the fact $x \notin \bigcup_{u \in \T^{\infty}} \partial \left( \tilde{T}(u) \right)$. Thus, we have proved \eref{end_pf} which concludes the proof of Lemma \ref{topo}.
\end{proof}

To complete the proof of Theorem \ref{well_def}, we use Dynkin's $\pi - \lambda$ theorem (Theorem 3.2 in \cite{Bil}). For any set $S$ of compact subspaces of $\tilde{T}$, denote $\sigma(S)$ the $\sigma$-algebra generated by $S$, so that $\sigma(\mathcal{O})$ is the usual Borel $\sigma$-algebra on $\tilde{T}$. To prove Theorem \ref{well_def}, it is sufficient to show that
\beq\label{s-algebra}
\sigma \left( \{ \tilde{T}(u); \, u \in \T^{\infty} \} \right) = \sigma(\mathcal{O}).
\eq
But $\sigma \left( \{ \tilde{T}(u); \, u \in \T^{\infty} \} \right)$ is a Dynkin system (since it is a $\sigma$-algebra). Moreover, since $\T^{\infty}$ is countable, Lemma \ref{topo} implies that 
\[ \mathcal{O} \subset \sigma \left( \{ \tilde{T}(u); \, u \in \T^{\infty} \} \right), \]
and Dynkin's theorem immediately implies \eref{s-algebra}.
\end{proof}

\subsection{Properties of the limit measure}\label{IMP}

We have seen that the occupation measure $\mu_n$ converges in distribution to a limit measure $\mu$, which satisfies $\mu(\tilde{T}(u))=\vert I^u \vert$ where $I^u$ is the interval marking the node $u$ in the fragmentation construction introduced in Section \ref{frag}. Moreover, for any node $u$, $\mu(\partial \tilde{T}(u)) = 0$. In this section, we determine additional properties of the measure $\mu$.

\begin{pro}\label{at_part}
The atomic part of $\mu$ is a.s. zero. That is, a.s. there is no point $x \in T$ s.t. $\mu( \lbrace x \rbrace) >0$.
\end{pro}

\begin{proof}
Define $T^{(n)} := \lbrace \tilde{T}(u); \, u \in \T^{\infty}, |u| = n \rbrace $, that is the set of triangles `` at height $n$ ". It suffices to show that a.s.
\beq\label{proof_at}
\sup_{\tau \in T^{(n)}} \mu(\tau) \rightarrow 0, \quad \mbox{as } n \rightarrow \infty.
\eq
Indeed, if there exists with positive probability some $x \in T$ such that $\mu(\{x \}) >0$, then, using the notation $T^n(x)$ introduced in the proof of Theorem \ref{well_def}, with positive probability
\[ \limsup_n \sup_{\tau \in T^{(n)}} \mu(\tau) \geq \limsup_n \mu (T_n(x)) \]
\[ \geq \limsup_n \mu( \{x \}) = \mu( \{x \}) > 0, \]
and therefore proving \eref{proof_at} is sufficient.

For this, we will use a branching process result. Notice that if $|u| = n$, then the law of $\mu(\tilde{T}(u))$ is $\mathcal{P}_1 \cdots \mathcal{P}_n$ where the $\mathcal{P}_i$ are i.i.d. random variables with distribution the first (or equivalently, any) marginal of a $\Dir_2(\frac{1}{2})$ distribution. We shall show that
\beq\label{proof_at'} 
\inf_{\tau \in T^{(n)}} -\log(\mu(\tau)) \rightarrow +\infty, \quad \mbox{as } n \rightarrow \infty,
\eq
which is equivalent to \eref{proof_at}.

Now the law of $\inf_{\tau \in T^{(n)}} -\log(\mu(\tau))$ is the law of the time of first birth at generation $n$ for a branching process with birth times $-\log(P_1)$, $-\log(P_2)$, $-\log(P_3)$ where $(P_1,P_2,P_3)$ has law $\Dir_2(\frac{1}{2})$ (and every vertex has exactly three children).

We define $\Phi$ to be the Laplace transform of the reproduction law:
\[\Phi(\theta):= \E\left[\sum_{i=1}^3 \exp(- \theta.(-\log(\mathcal{P}_i)))\right].\]
Thus $\Phi(\theta) = 3\E\left((\mathcal{P}_1)^{\theta}\right)$. Kingman proved in \cite{King} that if $a,\theta>0$ satisfy $\Phi(\theta)e^{\theta a} < 1$ then the first birth process $(B_n)$ satisfies $\liminf_n \frac{B_n}{n} \geq a$.

Now since $\Phi(\theta)$ tends to zero as $\theta$ tends to infinity, we can choose $\theta_0$ such that $\Phi(\theta_0) \leq (2e)^{-1}$ and taking $a=\theta_0^{-1}$ will give us the desired result. This is clearly enough to show \eref{proof_at'} (since $a>0$), and thus Proposition \ref{at_part} is proved.
\end{proof}

It is a well known fact that any Borel measure $\nu$ on $\R^d$ can be decomposed as $\nu = \nu_{Leb} + \nu_{atom} + \nu_{sing} $, where $\nu_{Leb}$ is absolutely continuous with respect to the Lebesgue measure on $\R^d$, $\nu_{atom}$ is a countable (weighted) sum of Dirac atoms, and $\nu_{sing}$ has no atoms and is singular with respect to the Lebesgue measure on $\R^d$. The previous theorem means that a.s. $\mu_{atom} = 0$. We seek additional information on $\mu$.

\begin{de}\label{HDimdef}
Let $M$ be a metric space, and $X \subseteq M$ a subspace of $M$. For any $d \geq 0$, we define the \emph{$d$-dimensional Hausdorff measure} $\mu_d$ of $X$ by
\[ \mu_d(X) = \lim_{ \varepsilon \rightarrow 0} \inf \sum_{i \in I} (\mathrm{diam}(U_i))^d,\]
where the infimum is taken over all countable coverings $(U_i)_{i \in I}$ of $X$ such that for any $i \in I$, $\mathrm{diam} (U_i) < \varepsilon$. This infimum is non decreasing as $\varepsilon$ decreases, thus the limit exists.
The \emph{Hausdorff dimension} $\dim_H$ of $X$ is then defined by
\[ \dm(X) = \sup \{ d \geq 0; 0 < \mu_d(X) < \infty \}. \]
\end{de}

\begin{thm}\label{HDim}
The limit measure $\mu$ is supported by a subset $S_{\nu}(\mu)$ of $\tilde{T}$ which satisfies
\[ \dm(S_{\nu}(\mu)) = \frac{2}{3 \E(-\log(P_1))} \mbox{ a.s.} ,\]
where $P=(P_1,P_2,P_3)$ is a splitting ratio with distribution $\nu$.
\end{thm}

Using a convexity inequality we get that
\[ \frac{2}{3 \E(-\log(P_1))} \leq \frac{2}{-3 \log(\E(P_1))} = \frac{2}{3 \log(3)} ,\]
and since in particular the latter quantity is strictly less than $2$ we get that $\mu$ is a.s. singular with respect to the Lebesgue measure, i.e. a.s. $\nu_{Leb} = 0$. Notice also that we have equality in the above when we take the special case $P_1 = \frac13$ a.s., that is $\nu = \nu^0$ as defined in \eref{nu0}.

\begin{proof}
We apply the second point of Corollary IV.b in \cite{Bar}. In our case, $c=3$, $(W_0,W_1,W_2)$ follows the $\Dir_2(\frac12)$ distribution, and $(L_0,L_1,L_2) = P = (P_1,P_2,P_3)$, where $P$ is a splitting ratio with distribution $\nu$. Barral's result states that the measure $\mu$ is supported by a set with Hausdorff dimension
\[ \dm(S(\mu)) = \frac{\E \left( \sum_{j=0}^{c-1} W_j \log W_j \right)}{\E \left( \sum_{j=0}^{c-1} W_j \log L_j \right)}\]
\[ = \frac{\E(W_0 \log(W_0))}{ \E(W_0) \E(\log(P_1))} ,\]
and the desired result follows from a computation.
\end{proof}


\begin{thebibliography}{10}

\bibitem{AM}
M.~Albenque and J.-F. Marckert.
\newblock Some families of increasing planar maps.
\newblock {\em Electron. J. Probab.}, 13:no. 56, 1624--1671, 2008.

\bibitem{CRT}
D.~Aldous.
\newblock Cambridge University Press, 1991.

\bibitem{Ald1}
D.~Aldous.
\newblock Tree-based models for random distribution of mass.
\newblock {\em J. Stat. Phys}, 73:625--641, 1993.

\bibitem{Ald2}
D.~Aldous.
\newblock Recursive self-similarity for random trees, random triangulations and
  brownian excursion.
\newblock {\em The Annals of Probability}, 22(2):pp. 527--545, 1994.

\bibitem{Bar}
J.~Barral.
\newblock Moments, continuité, et analyse multifractale des martingales de
  mandelbrot.
\newblock {\em Probability Theory and Related Fields}, 113:535--569, 1999.

\bibitem{Bil}
P.~Billingsley.
\newblock {\em {Probability and measure. 3rd ed.}}
\newblock {Chichester: John Wiley \& Sons Ltd.}, 1995.

\bibitem{Bon}
N.~Bonichon, S.~Felsner, and M.~Mosbah.
\newblock {Convex drawings of 3-connected plane graphs.}
\newblock {\em Algorithmica}, 47(4):399--420, 2007.

\bibitem{BBI}
D.~Burago, Y.~Burago, and S.~Ivanov.
\newblock {\em {A course in metric geometry.}}
\newblock {Providence, RI: American Mathematical Society (AMS)}, 2001.

\bibitem{CK}
N.~Curien and I.~Kortchemski.
\newblock Random non-crossing plane configurations: A conditioned galton-watson
  tree approach.
\newblock 2012.

\bibitem{Fek}
E.~Fekete.
\newblock {Branching random walks on binary search trees: convergence of the
  occupation measure.}
\newblock {\em ESAIM, Probab. Stat.}, 14:286--298, 2010.

\bibitem{Gil}
F.~Gillet.
\newblock {\em Etude d'algorithmes stochastiques et arbres}.
\newblock 2003.

\bibitem{King}
J.~Kingman.
\newblock {The first birth problem for an age-dependent branching process.}
\newblock {\em Ann. Probab.}, 3:790--801, 1975.

\bibitem{LeGall}
J.-F. {Le Gall}.
\newblock {Uniqueness and universality of the Brownian map}.
\newblock {\em ArXiv e-prints}, May 2011.

\bibitem{Mier}
G.~{Miermont}.
\newblock {The Brownian map is the scaling limit of uniform random plane
  quadrangulations}.
\newblock {\em ArXiv e-prints}, Apr. 2011.

\bibitem{Sch}
G.~Schaeffer.
\newblock {\em Conjugaison d'arbres et cartes combinatoires al{\'e}atoires}.
\newblock 1998.

\end{thebibliography}
\end{document}